\documentclass[leqno,twoside, 11pt]{amsart}
\usepackage{amssymb, latexsym, esint}
\usepackage{color}

\theoremstyle{plain}
\numberwithin{equation}{section} \numberwithin{figure}{section}

\newtheorem{theorem}{Theorem}[section]
\newtheorem{lemma}[theorem]{Lemma}
\newtheorem{proposition}[theorem]{Proposition}

\newtheorem{definition}[theorem]{Definition}
  \usepackage{epsfig} %For pictures: screened artwork should be set up with an 85 or 100 line screen
\usepackage{graphicx}  \usepackage{epstopdf}
\theoremstyle{definition}
\newtheorem{remark}[theorem]{Remark}

\begin{document}

\markboth{Pablo Ochoa and Julio Alejo Ruiz}
{Singular problems in Carnot groups}


%
%%%%%%%%%%%%%%%%%%%%%%%%%%%%%%%%%%%%%%%%%%%%%%%%%%%%%%%%%%%%%%%%%%%%

\title[Singular problems in Carnot groups]{Solving singular evolution problems in sub-Riemannian groups via deterministic games}

\author{Pablo Ochoa and Julio Alejo Ruiz}

\address{Universidad Nacional  de Cuyo-CONICET, Parque Gral. San Mart\'in\\
Mendoza, 5500, Argentina. \\
ocho\-pablo@gmail.com}

\address{Universidad Nacional  de Cuyo-CONICET, Parque Gral. San Mart\'in\\
Mendoza, 5500, Argentina.\\
julioalejoruiz@gmail.com}
\maketitle

\begin{abstract}
In this manuscript, we prove existence of viscosity solutions to singular parabolic equations in Carnot groups. We develop the  analysis by constructing appropriate deterministic games adapted to the algebraic and differential structures of Carnot groups. We point out that the proof of existence does not require comparison and it is based on an Arzela-Ascoli-type theorem. 
\end{abstract}

\section{Introduction}
In the last decades, there has been a special interest in the study of partial diffe\-rential equations in non-Euclidean frameworks. In this work, we shall be concerned with the development of  existence  results for  viscosity solutions of singular parabolic equations in Carnot groups via two-person deterministic games.

To motivate and comment about the main results,  consider a family of  surfaces $M_t \subset \mathbb{R}^{N}$, $t \geq 0$, given as the zero-level set of a function $u$:
$$M_t = \Big\{p \in \mathbb{R}^{N}: u(t, p)=0\Big\}.$$
If we are interested in the movement of $M_t$ by horizontal mean curvature,  then the function $u$  is a solution of the PDE:

\begin{equation}\label{MCFEn}
u_t(t, p)= \sum_{i,j=1}^{m_1}\Big( \delta_{ij}- \frac{X_iuX_ju}{\sum_{i=1}^{m_1}(X_iu)^{2}}\Big)X_iX_ju.
 \end{equation}This model is known as the horizontal mean curvature flow equation (see \cite{G} and  \cite{CGG}  for a derivation in Euclidean spaces, and  \cite{CC} and \cite{CD} for the corresponding discussion in  Carnot groups). 
The main result  of the paper is Theorem \ref{existenceb} where we establish the existence of viscosity solution to \eqref{MCFEn} by employing two-person deterministic games.  We point out that our result apply to a large  class of singular equations including \eqref{MCFEn}.

Existence and comparison results for Carnot groups are not largely availa\-ble compared to the Euclidean framework. We  highlight that even do one may write equations in Carnot groups in terms of the Euclidean gradient and Hessian of the unknown and try to apply the Euclidean theory (for instance from \cite{CIL}), this does not always work. Indeed, the above procedure may introduce degenerate points in the equations or may causes the loss of appropriate continuity. A classical example is the following: consider the sub-elliptic Laplacian operator in the Heisenberg group
$$\Delta_0:=X_1^{2}+X_2^{2}$$where $X_1$ and $X_2$ are given as in \eqref{frame H}.  In terms of Euclidean derivatives, the sub-elliptic Laplace equation may be written as:
$$-\textnormal{tr}\left( \begin{pmatrix}
1 & 0 & 2y \\
0 & 1 & -2x \\
2y & -2x & 4(x^{2}+y^{2})
\end{pmatrix}\nabla^{2}u \right)=0.$$Observe that the matrix:
$$\begin{pmatrix}
1 & 0 & 2y \\
0 & 1 & -2x \\
2y & -2x & 4(x^{2}+y^{2})
\end{pmatrix}$$is not uniformly elliptic for all $(x, y, z)$.  The interested reader may consult  \cite[Remark 1]{OR} where we exhibited an example for which the operator in terms of the Euclidean derivatives does not satisfies well-known assumptions on uniform continuity needed and largely used in the theory of Euclidean viscosity solutions. However, the comparison derived intrinsically in the sub-Rimannian setting in \cite{OR} applied to the given example.  We also refer the reader to Remark \ref{remar euc} in the present article.

 In sub-Riemannian structures, the singularity of the equation  may appear at points where part of the gradient (usually the so-called horizontal gradient) vanishes. This is not the case of the Euclidean context, where the singularity comes from points where the full gradient vanishes. This facet of sub-Riemannian structures yields complications in the study of singular equations.  In spite of these facts, there are some interested findings in Carnot groups.  For elliptic and uniformly elliptic equations in the Heisenberg group $\mathbb{H}$, comparison results are given in \cite{M}, based on a sub-elliptic version of the Euclidean Crandall-Ishii Lemma (see  \cite{CIL} for details). At this point we  quote the works \cite{Bardi-M} and \cite{Manucci} where the authors propose various form of partial nondegeneracy to weaken the uniform ellipticity assumption and  apply their results to some sub-elliptic second order equations.  Also,  in \cite{Bi}, results related to  infinite harmonic functions in the Heisenberg group were esta\-blished.  In the general case of Carnot groups, we find the work \cite{Bi3} for the p-Laplacian operator.  In the setting of vector fields in $\mathbb{R}^{N}$ (with its standard group structure), we  refer the reader to the paper \cite{BBM}. Regarding parabolic equations in Carnot groups, we mention \cite{Bi2}, where comparison results for \textit{admissible}  operators in the Heisenberg group where obtained. The word \textit{admissible} refers to continuous and proper ope\-rators $\mathcal{F}=\mathcal{F}(t, p, u, \eta, \mathcal{X}
)$ which satisfy the following: for each $t \in [0,T]$, there is a modulus of continuity $\omega: [0,\infty]\to [0,\infty]$ so that
\begin{eqnarray*}
\mathcal{F}(t, q,r, \tau \eta,\mathcal{Y})- \mathcal{F}(t,p,r, \tau \eta,\mathcal{X}) \leq \omega \big( d_C(p,q)+ \tau \|\eta\|^{2}+ \|\mathcal{X}-\mathcal{Y}\|\big),
\end{eqnarray*}where $\tau >0$. Hence, the results are not valid for singular equations. In \cite{EM}, the author provides existence and uniqueness results for the Gauss curvature flow equation of graph in Carnot group in unbounded domains, generalizing the available results in the literature. For singular equations, in \cite{Bi1} it was covered the case of the parabolic p-Laplacian, but the structure of this equation in largely used in the derivation of the uniqueness principle. A remarkable progress was done in \cite{FLM1}, where the authors proved existence and a comparison principle (needed for existence) for radially symmetric viscosity solutions for the horizontal mean curvature flow equation in the first-order Heisenberg group. The symmetry refers to solutions $u= u(t,p_1,p_2,p_3)$ for which it holds:
\begin{eqnarray*}
u(t,p_1,p_2,p_3) = u(t, \tilde{p}_1,\tilde{p}_2, p_3) \mbox{ whenever }p^{2}_1+p_2^{2}= \tilde{p}_1^{2}+\tilde{p}_2^{2}.
\end{eqnarray*}

In \cite{FLM1}, the existence of viscosity solutions is obtained via deterministic games. Here, we extend the findings of  \cite{FLM1} to more general singular parabolic equations, by adapting the reasoning of \cite{KS} and \cite{Kasai} to the structure of Carnot groups. We point out that, unlike \cite{FLM1}, our existence result does not require comparison and it is based on an Arzela-Ascoli-type theorem from \cite{MPR}.  As a final comment, we mention that the uniqueness of the solutions constructed here constitutes an open problem.

The organization of the paper is as follows. In Section 2 we introduce the type of equations we study in the paper and the main result concerning exitence of solutions. In the end of the section, we  provide some applications of our results. In Section 3 we shall introduce the necessary background and notation on Carnot groups as well as the notion of viscosity solutions by means of parabolic jets. In the next Section 4, we provide the proof of the existence result. We close the paper  with an Appendix where we prove a technical lemma needed in the analysis of the deterministic games.

\section{Main Result of the paper and assumptions}

In this work, existence of solutions is discussed for initial-value problems of the form:

\begin{equation}\label{problem}
  \left\lbrace
  \begin{array}{l}
       u_t + \mu u + \mathcal{F}(t, p, \nabla_{\mathcal{G}, 0}u, \nabla^{2, *}_{\mathcal{G}, 0}u) =0, \textnormal{ on } (0, T) \times \mathbb{G}, \\
   \qquad \qquad \qquad \qquad \qquad \,\,\,  u(0, p) = \psi(p), \textnormal{ with } p \in \mathbb{G}.\\
  \end{array}
  \right.
\end{equation}Here,  $T> 0$ is fixed   and  $\mu \geq 0$ is a parameter. We kindly refer the reader to Section \ref{prelims}  for notation and basic definitions involving Carnot groups. We assume that $\psi$ satisfies:

 \begin{itemize}
\item[(E)] $\psi  \in BC(\mathbb{G})$ and for each $\delta >0$ there are smooth approximations $\psi^{+}_\delta, \psi^{-}_\delta$ with bounded right and left-invariant horizontal derivatives of first and second order  so that:
$$\psi-\delta \leq \psi^{-}_\delta \leq \psi \leq  \psi^{+}_\delta \leq \psi+\delta.$$ 
\end{itemize}In addition, the operator $\mathcal{F}$ satisfies the following assumptions: 

\begin{enumerate}

    \item[\textbf{($\mathcal{F}1$)}] $\mathcal{F} : [0, T]\times \mathbb{G}\times \left(\mathbb{R}^{m_1} \backslash \lbrace 0 \rbrace\right) \times \mathcal{S}^{m_1}(\mathbb{R}) \rightarrow \mathbb{R}$ is continuous.
    
    \item[\textbf{($\mathcal{F}2$)}] $\lambda_0:=\sup\limits_{\eta} \vert \mathcal{F}(t, p, \eta, \mathcal{O}) \vert < \infty$.
    
    \item[\textbf{($\mathcal{F}3$)}] There exists a positive constant $\lambda_1$ such that:
    \begin{equation*}
        \mathcal{F}(t, p, \eta, \mathcal{X}) - \mathcal{F}(t, p, \eta, \hat{\mathcal{X}}) \leq \dfrac{\lambda_1^2}{2} \mathcal{E}^+(\hat{\mathcal{X}} - \mathcal{X}),    \end{equation*}where:
        $$\mathcal{E}^+(\mathcal{X}):= \max\left\lbrace 0, \text{ maximum eigenvalue of }\mathcal{X}\right\rbrace,$$and there exists a modulus of continuity $\omega=\omega(r)=O(r)$ as $r \to 0^{+}$ such that:
\begin{eqnarray}
\mathcal{F}(t,p,\eta, \mathcal{X})- \mathcal{F}(s,q,\eta,  \mathcal{X}) \leq \omega\big( |s-t|+ |p \cdot q^{-1}|_\mathbb{G}\big), 
\end{eqnarray}for all $(t,p), (s,q)\in [0, T]\times \mathbb{G}$,  $\eta \in \mathbb{R}^{m_1}\setminus \{0\}$,  and $\mathcal{X}\in S^{m_1}(\mathbb{R})$.
   
    \item[\textbf{($\mathcal{F}4$)}] For any $r, R > 0$, there exists a modulus of continuity $\omega_{r, R}$ such that:
    \begin{equation*}
        \mathcal{F}(t, p, \hat{\eta}, \mathcal{X}) - \mathcal{F}(t, p, \eta, \mathcal{X}) \leq \omega_{r, R}(\Vert \hat{\eta} - \eta \Vert), \textnormal{ if } \Vert \hat{\eta} \Vert, \Vert \eta \Vert \geq r, \,\Vert \mathcal{X} \Vert \leq R. 
    \end{equation*}
    
    \item[\textbf{($\mathcal{F}5$)}] $ \mathcal{F}_*(t, p, 0, \mathcal{O}) = \mathcal{F}^*(t, p, 0, \mathcal{O}) =0$.
\end{enumerate}

The main theorem of the paper is the following existence result.

 \begin{theorem}\label{existenceb} Assume (E) and $(\mathcal{F}1)-(\mathcal{F}5)$.  Then there is a  viscosity solution $u \in BUC([0, T] \times \mathbb{G})$ to \eqref{problem}.
 \end{theorem}
 
 \
 
\begin{remark} Assumption (E) on the datum may be obtained for $\psi \in BC(\mathbb{G})$ which are constant outside a compact set. That is the setting of \cite{FLM1}.
\end{remark} 

\begin{remark}\label{remar euc} In the text, assumptions ($\mathcal{F}3$) and ($\mathcal{F}4$) will be usually  applied to:
$$\eta=\nabla_{\mathcal{G}, 0}\phi(p) \quad \text{ and }\quad \mathcal{X}=\nabla^{2, *}_{\mathcal{G}, 0}\phi (p),$$for some smooth $\phi$ (we do not write the dependence on $t$). In terms of Euclidean derivatives, there are matrix fields $A$ and $M$ (see \cite[Lemma 3.2]{Bi1}) so that:
$$\nabla_{\mathcal{G}, 0}\phi(p) = A(p)\nabla \phi(p) \quad \text{ and } \quad  \nabla^{2, *}_{\mathcal{G}, 0}\phi (p)=A (p)\nabla^{2}\phi(p)A(p) + M(p),$$where $\nabla \phi(p)$ and  $\nabla^{2}\phi(p)$ denotes the Euclidean gradient and Hessian of $\phi$ at $p$. Hence in order to consider sub-elliptic equations in terms on Euclidean derivatives, it is natural to  introduce the operator:
$$\mathcal{G}(t, p, \eta_E, X_E):=\mathcal{F}(t, p, A(p)\eta_E, A (p)X_EA(p) + M(p)), \quad \eta_E \in \mathbb{R}^{N}, X_E \in S^{N}(\mathbb{R})$$which necessarily depends  on $t$ and $p$ (so the existence result from \cite{Kasai} is not applied). Moreover, hypothesis ($\mathcal{F}4$) does not imply, in general, \textbf{(F4)} from \cite{Kasai}, since the norm of $A(p)$ is not uniformly bounded in $\mathbb{G}$. 
\end{remark}
 
% \begin{remark}The viscosity solution from Theorem \ref{existenceb} is unique among the class of symmetric solutions which agree outside  a \textit{compact cylinder} $K \times [0, T]$, for  compact $K \subset \Omega.$ See Section 7.
% \end{remark}
 
 \subsection{Some Applications} \subsubsection{Mean curvature flow equation} If $M \subset \mathbb{G}$ is a smooth hypersurface, we define $\Sigma(M)$ as the set of cha\-racteristic points of $M$, that is, the points $p \in M$ where the horizontal distribution at $p$ is contained in the tangent space of $M$ at $p$. The horizontal mean curvature flow is the flow $t \to M_t$ in which each point $p(t)\notin \Sigma(M_t)$ in the evolving surface moves along the horizontal normal with speed given by the horizontal  mean curvature. The equation, outside the characteristic set, may be written as \eqref{MCFEn}. For (\ref{MCFEn}), the singular operator  is given by:
 
 \begin{eqnarray*}
\mathcal{F}_{MCF}(\eta, \mathcal{X})= - \mbox{tr}\Big[\Big( I - \frac{\eta \otimes \eta}{\|\eta\|^{2}}\Big)\mathcal{X}\Big],
\end{eqnarray*}for $\eta \in \mathbb{R}^{m_1}\setminus \left\lbrace 0 \right\rbrace $ and $\mathcal{X}\in S^{m_1}(\mathbb{R})$. Hence, existence of solutions follows from Theorem \ref{existenceb}.

 \subsubsection{Parabolic infinite Laplacian} The infinite Laplacian is connected with the problem of finding minimal Lipschitz extensions, called absolute minimizers. In a Carnot group, we say that a Lipschitz function $u$ in $\overline{\Omega}$ is an absolute minimizer if for every $V \subset \Omega$ and every Lipschitz function $h$ in $V$ such that $u=v$ on $\partial V$, there holds:
 $$\|\nabla_{\mathcal{G},0}u\|_{L^{\infty}(V)}\leq \|\nabla_{\mathcal{G}, 0}h\|_{L^{\infty}(V)}.$$It has been established independently in \cite{BC} and \cite{W} that absolutely minimizers are viscosity solutions of the infinite Laplace equation. We consider the parabolic and normalized counterpart:
 \begin{eqnarray}\label{infty}
 \partial_t u(t,p)= \frac{1}{\sum_{i=1}^{m_1}(X_iu)^{2}}\sum_{i, j=1}^{m_1}X_iuX_jX_iX_ju.
\end{eqnarray}  Hence, the singular operator is:
 \begin{eqnarray*}
\mathcal{F}_{PIL}(\eta, \mathcal{X}) = - \frac{1}{\|\eta\|^{2}}\langle \mathcal{X}\eta, \eta \rangle,
\end{eqnarray*}for $\eta \in \mathbb{R}^{m_1}\setminus \left\lbrace 0 \right\rbrace $ and $\mathcal{X}\in S^{m_1}(\mathbb{R})$. Hence, existence is derived from our analysis.

 \section{Preliminaries}\label{prelims}
 
 \subsection{Carnot groups}\label{sec CG}
 
Let $\mathbb{G}$ be a connected and simply connected Lie group, whose Lie algebra $\mathcal{G}$ is real and N-dimensional. We say that $\mathbb{G}$ is a Carnot group of step $l\geq 1$ if  $\mathcal{G}$ has a stratification, that is, there exist vector spaces $V_1$, ..., $V_l$ such that:
 \begin{eqnarray*}
 \mathcal{G}= V_1 \oplus \cdots \oplus V_l, \quad [V_1, V_i]=V_{i+1}, \quad 1\leq i \leq l-1, \quad [V_i, V_l]=0, \quad i=1, ..., l.
 \end{eqnarray*}Here, $[V_1, V_i]$ stands for the linear subspace generated by the vectors $[X, Y]$, where $X \in V_1$ and $Y \in V_i$. In particular, $\mathbb{G}$ is nilpotent. Choose a Riemannian metric with respect to which the $V_i$ are mutually orthogonal. Let $m_i = \mbox{dim } V_i$, for $i=1, ..., l$ and consider $h_r = m_1 + \cdots + m_r$, $0 \leq r \leq l$, with $h_0=0$. Choose an orthonormal  basis of $V_i$ of left-invariant vector fields $X_{j}$, $j=h_{i-1}+1, ..., h_i$. Thus, the dimension of $\mathbb{G}$ as a manifold is $N= h_l= m_1+\cdots m_l$. The exponential map  $\mbox{exp}: \mathcal{G}\to \mathbb{G}$ is a global diffeomorphism and may be used to define exponential coordinates $\varphi$ in $\mathbb{G}$ as follows: any $p \in \mathbb{G}$
 may be written uniquely as
 $$p= \mbox{exp}\big(p_1X_{1}+ \cdots + p_N X_{N} \big),$$and thus we may put $\varphi(p)=(p_1,..., p_N)$. In this way, we identify $\mathbb{G}$ with $(\mathbb{R}^{N}, \cdot)$, where the group law $\cdot$ is given by the Campbell-Hausdorff formula \cite{BLU} as:
 \begin{eqnarray}\label{operation}
 \varphi_i(p\cdot q) = \varphi_i(p) + \varphi_i(q) + R_i(p, q), \qquad i=1, ..., N,
 \end{eqnarray}where $R_i$ depends only on $\varphi_k$ for $k < i$.  In what follows, we just write $p_i$ for $\varphi_i(p)$. We sometimes use the notation:
 $$p= (p_{1, 1}, ..., p_{l, m_l}).$$
 
 The first layer $V_1$ spanned by the vector fields $X_1, ..., X_{m_1}$ plays an important role in the theory  and it is called the horizontal distribution. Thus, for every $p$ in $\mathbb{G}$:
 \begin{eqnarray*}
V_{1,p}= \mbox{ span }\big\{X_{1,p}, ..., X_{m_1, p}\big\}.
 \end{eqnarray*}
 
% We will also consider the homogeneous dimension of $\mathbb{G}$, denoted by $Q$, and defined by:
% \begin{eqnarray*}
% Q= \sum_{i=1}^{l}i\cdot\mbox{dim }V_i.
% \end{eqnarray*}

 \subsubsection{Metric structure on $\mathbb{G}$} If $\gamma: I=[0,1] \to \mathbb{G}$ is an absolutely continuous curve in $\mathbb{G}$ that satisfies:
$$\gamma'(t) \in V_{1, \gamma(t)}, \textnormal{ for a. e.  }t \in I,$$we call $\gamma$ a horizontal path. We define the Carnot-Carath\'eodory distance on $\mathbb{G}$ by:
\begin{eqnarray*}
d_C(p,q)= \inf\Big\{\|\gamma'(t)\|:\gamma'(t) \in V_{1, \gamma(t)}, \,\forall t \in I, \,\gamma(0)=p,\,\gamma(1)=q \Big\},
\end{eqnarray*}where $\|\cdot \|$ is the norm induced by the Riemann structure on $\mathcal{G}$. Since Carnot groups satisfy the H\"{o}rmander's condition, we get by Chow's Theorem, that $d$ is well-defined.  It is well-known that the topology induced by $d_C$ is equivalent to the Euclidean topology. However, $d_C$ is not bi-Lipschitz equivalent to the 	Euclidean distance (see \cite{NS}).

 \subsubsection{Calculus on Carnot groups} Consider the Carnot group $\mathbb{R}\times \mathbb{G}$ where we add $\partial/\partial t$ to the horizontal frame as $X_0$. For any $1\leq k \leq l$, we say that $u: \mathbb{R}\times \mathbb{G}\to \mathbb{R}$ belongs to $\mathcal{C}^{k}_{sub}$ if it is continuous and:
 $$X^{I}u \text{ is continuous for all } I=(i_0, i_1, ..., i_{N}) \text{ so that } d(I)=i_0 d_0+ i_1d_1+\cdots + i_Nd_N \leq k$$where $d_0=1$ and  $d_m=j$ if the corresponding vector field belongs to $V_j$, $j=1, ..., l.$

% For a smooth function $u: \mathbb{G}\to \mathbb{R}$,  the Euclidean partial derivatives $\partial u/\partial x_i$ are always understood as derivatives with respect to the exponential coordinates in $\mathbb{G}$. We denote the Euclidean gradient at $p$ by $\nabla u (p)$.
The full (spacial) gradient with respect to the Carnot frame $\{X_1, ..., X_N\}$ will be given by:
 $$\nabla_{\mathcal{G}}u(p)= \sum_{i=1}^{N}(X_{i, p}u) X_{i, p}, \qquad p \in \mathbb{G}.$$We shall also consider the horizontal and second order horizontal gradients of $u$:
\begin{eqnarray*}
\nabla_{\mathcal{G}, 0}u(p)= \sum_{i=1}^{h_1}(X_{i,p}u)X_{i, p}, \quad \text{and }\,\nabla_{\mathcal{G}, 1}u(p)= \sum_{i=h_1+1}^{h_2}(X_{i,p}u)X_{i, p}.
\end{eqnarray*}
%Thus, $\nabla_{\mathcal{G}, 0}u(p)$ is the projection of the full gradient $\nabla_{\mathcal{G}}u(p)$ onto the horizontal distribution $V_{1,p}$. Introducing the second order horizontal gradient of $u$:
The symmetrized horizontal Hessian matrix, denoted by $\nabla^{2, *}_{\mathcal{G}, 0}u$, has the entries:
\begin{eqnarray*}
\nabla^{2, *}_{\mathcal{G}, 0}u_{ij}= \frac{1}{2}\big( X_iX_ju+X_jX_iu \big), \quad i,j=1,..., m_1,
\end{eqnarray*}The stratified  Taylor expansion of a $\mathcal{C}^{2}_{sub}$-function  $u$ at $(t, p)\in \mathbb{R}\times \mathbb{G}$ reads as (see  \cite[Theorem 1.42]{FS} or Exercise 6 in \cite[Chapter 20]{BLU}):
\begin{equation}\label{stratified taylos formula}
\begin{split}
u(s, q)&=u(t, p)+u_t(t, p)(s-t)+ \left\langle \nabla_{\mathcal{G}, 0}u(t, p), (p^{-1}\cdot q)_1\right\rangle + \left\langle \nabla_{\mathcal{G}, 1}u(t, p), (p^{-1}\cdot q)_2\right\rangle\\ & \qquad \qquad \qquad + \frac{1}{2}\left\langle \nabla^{2, *}_{\mathcal{G}, 0}u(t, p)(p^{-1}\cdot q)_1, (p^{-1}\cdot q)_1\right\rangle + o(d_C(p,q)^{2} + |s-t|).
\end{split}
\end{equation}Here,  $(p^{-1}\cdot q)_1$ and $(p^{-1}\cdot q)_2$ denote the projection of $p^{-1}\cdot q$ onto $V_1$ and $V_2$, respectively. We remark that if instead of choosing a left-invariant frame, we consider a right-invariant basis of the Lie algebra, then we may also define right horizontal derivatives of first and second order (see \cite{FS}). 

\subsubsection*{Example 1: Euclidean spaces} The simplest example of a Carnot group is the Euclidean space with the usual norm $(\mathbb{R}^{N}, |\cdot|)$. This is a Carnot group of step $1$.

\subsubsection*{Example 2: Heisenberg group} One of the   most familiar Carnot groups is the Heisenberg group $\mathbb{H}$, whose background manifold is $\mathbb{R}^{2n+1}$. Given two points $p=(p_1, ..., p_{2n+1})$ and $q=(q_1, ..., q_{2n+1})$ in $\mathbb{H}$, we define a group operation by:
\begin{eqnarray*}
p\cdot q= \Big( p_1+q_1, ..., p_{2n}+q_{2n}, p_{2n+1}+q_{2n+1}+ \frac{1}{2}\sum_{i=1}^{2n}(p_iq_{i+n}-p_{i+n}q_i)\Big)
\end{eqnarray*}The standard basis of left invariant vector fields of the Heisenberg Lie algebra, denoted by $\mathcal{H}$, is given by:
\begin{eqnarray}\label{frame H}
&&X_j  = \partial_j -\frac{p_{n+j}}{2}\partial_{2n+1}, \qquad j=1, ..., n.\\
&&Y_j = \partial_{j+n} + \frac{p_j}{2}\partial_{2n+1}, \qquad j=1, ..., n.\\
&&T = \partial_{2n+1}.
\end{eqnarray}Observe that:
$$[X_j, Y_j]=T, \quad j=1, ..., n.$$
Thus, the vector fields $X_1, ..., X_n, Y_1, ..., Y_n$ satisfies the H\"{o}rmander's condition. In this case, the stratification of the Lie algebra is given by:
$$\mathcal{H}= V_1 \oplus V_2,$$where $V_1 = \mbox{span}\{X_1, ..., X_n, Y_1, ..., Y_n\}$ and $V_2= \mbox{span}\{T\}$. Hence, $\mathbb{H}$ is a step-2 group. For applications of the Heisenberg group to quantum me\-chanics, the inte\-rested reader may consult the monograph \cite{Ov}. 

\subsubsection*{Example 3: Engel group} The Engel group $\mathbb{E}^{4}$ is the Carnot group whose Lie algebra is:
$$\mathcal{G}= V_1 \oplus V_2 \oplus V_3,$$where $V_1= \mbox{span}\{X_1, X_2\}$, $V_2= \mbox{span}\{X_3\}$ and $V_3=\mbox{span}\{X_4\}$, with the Carnot frame:
\begin{equation}\label{frame Engel}
\begin{split}
X_1&=\partial_1 - \frac{p_2}{2}\partial_3 -\Big( \frac{p_3}{2}+ \frac{p_2}{12}(p_1 + p_2)\Big)\partial_4\\
X_2&=\partial_2 + \frac{p_1}{2}\partial_3 -\Big( \frac{p_3}{2}- \frac{p_1}{12}(p_1 + p_2)\Big)\partial_4\\
X_3&= \partial_3 +\frac{1}{2}(p_1+p_2)\partial_4\\
X_4&=\partial_4,
\end{split}
\end{equation}$p=(p_1, ..., p_4).$ The Engel group is a Carnot group of step $3$. For more details on the Engel group see \cite{Mst}.

\subsection{Jets and viscosity solutions in Carnot groups} We shall recall the definition of parabolic jets in Carnot groups and the notion of viscosity solutions.

\subsubsection{Parabolic jets} Let $u$ be an upper-semicontinuous function in $[0, T]\times \mathbb{G}$. We define the parabolic superjet of $u$ at the point $(t,p)$ as:
\begin{eqnarray*}
&&\mathcal{P}^{2, +}u(t, p) = \Big\{(a, \eta, \xi, \mathcal{X}) \in \mathbb{R}\times \mathbb{R}^{m_1} \times \mathbb{R}^{m_2}\times S^{m_1}(\mathbb{R})\quad \mbox{such that } \\ &&\quad \qquad \, u(s, q) \leq u(t, p)+ a(s-t)+ \left\langle \eta, (p^{-1}\cdot q)_1 \right\rangle + \left\langle \xi, (p^{-1}\cdot q)_2 \right\rangle  \\ && \qquad\qquad\quad + \frac{1}{2}\left\langle \mathcal{X}(p^{-1}\cdot q)_1,  (p^{-1}\cdot q)_1 \right\rangle + o\big(d_c(p, q)^{2}+ |s-t|\big)\,\,\text{as } (s, q)\to (t,p)\Big\}.
\end{eqnarray*}Similarly, if $v$ is lower semicontinuous in  $[0, T]\times \mathbb{G}$, we define the parabolic subjet $\mathcal{P}^{2, -}v(t, p).$ It is known that parabolic jets may be seen as appropriate derivatives of test functions touching the given function by above or below. More precisely, if $u$ us upper semicontinuous, let us consider:
\begin{eqnarray*}
&&\mathcal{K}^{+, 2}u( t, p)=\Big\{\big(\phi_t(t,p), \nabla_{\mathcal{G},0}\phi(t,p),  \nabla_{\mathcal{G},1}\phi(t,p),  \nabla^{2, *}_{\mathcal{G}, 0}\phi(t,p)\big) \mbox{ so that $\phi$ is } \mathcal{C}^{2}_{sub}\\&& \quad  \qquad \qquad \qquad\mbox{and } (u-\phi)(s,q)\leq (u-\phi)(t,p) \mbox{ for all }(s,q) \mbox{ close to }(t,p)\Big\}
\end{eqnarray*}and similarly define $\mathcal{K}^{-, 2}v( t, p)$ for test function touching the lower semicontinuous function $v$ from below. Hence it follows that:
\begin{eqnarray}\label{IDENT}
\mathcal{P}^{2, +}u(t,p)= \mathcal{K}^{2,+}u(t,p)
\end{eqnarray}and:
\begin{eqnarray*}
\mathcal{P}^{2, -}v(t,p)= \mathcal{K}^{2,-}v(t,p).
\end{eqnarray*}Finally, we shall also consider the theoretic closure of the sets defined above. We define $\overline{\mathcal{P}}^{2, +}u(t,p)$ as the set of  $(a, \eta, \xi, \mathcal{X})$ in $\mathbb{R}\times \mathbb{R}^{m_1}\times \mathbb{R}^{m_2}\times S^{m_1}(\mathbb{R})$ so that there exists a sequence $(t_n, p_n, a_n, \eta_n, \xi_n,  \mathcal{X}_n)$ converging to $(t,p,a,\eta, \mathcal{X})$ satisfying $(a_n, \eta_n, \mathcal{X}_n)\in \mathcal{P}^{2, +}u(t_n, p_n)$ for all $n$. In a similar way, we define $\overline{\mathcal{P}}^{2, -}v(t,p)$.

\subsubsection{Viscosity solutions} Let $\Omega \subset \mathbb{G}$ be a domain. Observe that in the following definition, the operator $\mathcal{F}$ does not depend on $\xi \in \mathbb{R}_{m_2}$. 

\begin{definition}\label{definition} An upper semicontinuous function $u: [0,T]\times \overline{\Omega} \to \mathbb{R}$ is a subsolution to the equation:
\begin{equation}\label{eq w}
u_t+\mathcal{F}(t, p, u, \nabla_{\mathcal{G}, 0}u, \nabla_{\mathcal{G}, 0}^{2, *}u)=0
\end{equation}in $(0,T)\times \Omega$ if for every $(t,p)\in (0,T)\times \Omega$ and every subjet $(a, \eta,  \xi, \mathcal{X}) \in \overline{\mathcal{P}}^{2, +}u(t,p)$ there holds:
\begin{eqnarray}\label{subsolution}
a + \mathcal{F}_{*}(t,p, u(t,p), \eta, \mathcal{X}) \leq 0.
\end{eqnarray}Here, the subscript $*$ stands for the lower semicontinuous envelope of $\mathcal{F}$. Similarly, we say that a lower semicontinuous function $v: [0,T]\times \overline{\Omega}  \to \mathbb{R}$ is a supersolution to the equation (\ref{eq w}) if  for every $(t,p)\in (0,T)\times \Omega$ and every superjet $(a, \eta, \xi,  \mathcal{Y}) \in \overline{\mathcal{P}}^{2, -}v(t,p)$ there holds:
\begin{eqnarray*}
a + \mathcal{F}^{*}(t,p, u(t,p), \eta,  \mathcal{Y}) \geq 0.
\end{eqnarray*}Analogously, the superscript $*$ stands for the upper semicontinuous envelope of $\mathcal{F}$. Finally, we say that a continuous function is a viscosity solution to (\ref{eq w}) if it is a viscosity sub and supersolution.
\end{definition}

\section{Deterministic games: proof of Theorem \ref{existenceb}}\label{existence sect}

In this section we  employ two-person deterministic games  to get existence of  solutions to:

\begin{equation*}
  \left\lbrace
  \begin{array}{l}
       u_t + \mu u + \mathcal{F}(t, p, \nabla_{\mathcal{G}, 0}u, \nabla^{2, *}_{\mathcal{G}, 0}u) =0, \textnormal{ on } (0, T) \times \mathbb{G}, \\
   \qquad \qquad \qquad \qquad \qquad \,\,\,  u(0, p) = \psi(p), \textnormal{ with } p \in \mathbb{G}.\\
  \end{array}
  \right.
\end{equation*}under the assumptions listed in Theorem \ref{existenceb}.

\begin{remark}\label{remarkhipotesysofFandH} From \textbf{($\mathcal{F}2$)} and \textbf{($\mathcal{F}3$)}, we conclude that $\mathcal{F}$ has at most linear growth (and at least linear decay). In fact, there exists $C=C(\lambda_0, \lambda_1)$ such that:
\begin{equation*}
    \vert \mathcal{F}(t, p, \eta, \mathcal{X}) \vert \leq C (1 + \Vert \mathcal{X} \Vert), \textnormal{ for } (t, p, \eta, \mathcal{X}) \in [0, T] \times \mathbb{G}\times (\mathbb{R}^{m_1} \backslash \lbrace 0 \rbrace) \times \mathcal{S}^{m_1}(\mathbb{R}).
\end{equation*}
In addition,  from \textbf{($\mathcal{F}3$)}, $\mathcal{F}$ is (degenerate) elliptic, since $\mathcal{F}(\cdot, \hat{\mathcal{X}}) \leq \mathcal{F}(\cdot, \mathcal{X})$ for $\hat{\mathcal{X}} \geq \mathcal{X}$.
\end{remark}

 We first describe the setting of the game. There are two players, Player I and Player II.  Let $T>0$ be the maturity of the game.  For each $\varepsilon \in (0, 1)$ let $m$ be the number of steps:
\begin{equation*}
    m := \left[\dfrac{T}{\varepsilon^{2}}\right],    
\end{equation*}where $[\cdot]$ is the integer part function. The players' choices are the following:

\begin{itemize}
\item the initial position is $p_0=p$ at $t_0=0$;
\item Player I chooses a pair $(\eta_0, \mathcal{X}_0) \in \left( \mathbb{R}^{m_1} \backslash \lbrace 0 \rbrace \right) \times \mathcal{S}^{m_1}\left( \mathbb{R} \right)$, with
\begin{equation*}
    \Vert \eta_0 \Vert \leq \varepsilon^{-1/4}, \qquad \Vert \mathcal{X}_0 \Vert \leq \varepsilon^{-1/2};
\end{equation*}
\item for these choices of Player I, Player II chooses a horizontal direction: \begin{equation*}
    q_0=(\nu_0, 0) \in \mathbb{G}, \textnormal{ with } \nu_0 \in \mathbb{R}^{m_1} \textnormal{ and } \vert q_0 \vert_{\mathbb{G}} \leq \varepsilon^{-1/4},
\end{equation*}where the Carnot gauge is defined as:
\begin{equation}\label{Heisenberggauge}
\vert p\vert_\mathbb{G}:= \sum_{j=1}^{l}\sum_{i=1}^{m_j}|p_{i, j}|^{1/j} \approx d_C(p, 0).
\end{equation}
\item Player I moves from $p_0$ to $p_1:=p_0 \cdot \delta_\varepsilon(q_0)$, where the dilatation $\delta_\varepsilon$ is given by $[\delta_\varepsilon(p)]_{i,j}=\varepsilon^{i}p_{i, j}$;
\item the above steps are repeated $m$ times;
\item at the maturity time $T$, Player I is at the final position $p_m$ and pays to Player II the amount:
\begin{equation}\label{pay off}
\left( \dfrac{1}{1+\mu\varepsilon^2} \right)^{m} \psi(p_{m}) + \sum\limits_{i=0}^{m -1} \left( \dfrac{1}{1+\mu\varepsilon^2} \right)^{i+1} R^{\varepsilon}(T-i\varepsilon^{2}, p_i, q_i, \eta_i, \mathcal{X}_i),
\end{equation}where $R^{\epsilon}$ is the running cost  defined in \eqref{running cost}, and $p_i$, $q_i, \eta_i$ and $\mathcal{X}_i$ are the  choices of the players at the $i$-th step. 

\end{itemize}

The value $u^{\varepsilon}$ of the game is obtained by considering that Player I has the objective of minimize \eqref{pay off} and Player II wants to maximize it. 

We extend now the set up of the game to any maturity $t$ and we formalize the definition of $u^{\varepsilon}$. Take $t \in [0, T]$ and consider the partition:
\begin{equation}\label{eqpartition}
    [0, T] = \lbrace 0 \rbrace \cup \left( \bigcup\limits_{k=1}^{m} \left( (k-1)\varepsilon^2, k\varepsilon^2 \right] \right).
\end{equation}Then if $t \neq 0$, there is a unique $k_t \in \lbrace 1, 2, \ldots, m \rbrace$ such that $t \in \left( (k_t-1)\varepsilon^2, k_t\varepsilon^2 \right]$. 

If $t=0$, we define:
\begin{equation*}
    u^{\varepsilon}(t, p) := \psi(p).
\end{equation*}

When $t \neq 0$, we have:
\begin{equation}\label{DPP}
    u^{\varepsilon}(t, p) := \inf_{\eta_{0}, \mathcal{X}_{0}} \sup_{q_{0}} \ldots \inf_{\eta_{k_t-1}, \mathcal{X}_{k_t-1}} \sup_{q_{k_t-1}} \left\{ \left( \dfrac{1}{1+\mu\varepsilon^2} \right)^{k_t} \psi(p_{k_t}) + \sum\limits_{i=0}^{k_t -1} \left( \dfrac{1}{1+\mu\varepsilon^2} \right)^{i+1} R^{ \varepsilon}(t-i\varepsilon^{2}, p_i, q_i, \eta_i, \mathcal{X}_i) \right\},
\end{equation}where:
\begin{equation}\label{running cost}
    R^{\varepsilon}(t, p_j, q_j, \eta_j, \mathcal{X}_j) := -\varepsilon \langle \eta_j, \nu_j \rangle -\dfrac{\varepsilon^2}{2} \left\langle \mathcal{X}_j \nu_j, \nu_j \right\rangle - \varepsilon^2 \mathcal{F} \left(t, p_j, \eta_j, \mathcal{X}_j \right).
\end{equation}

We introduce a rigorous definition for the values of the game based on the Dynamic Programming Principle. 

\begin{definition}For $\psi \in BC(\mathbb{G})$, define inductively:
$$u^{\varepsilon}(0, p)=\psi(p)$$and for $t \in ((k_t-1)\varepsilon^{2}, k_t\varepsilon^{2}]$:
$$u^{\varepsilon}(t, p)=\frac{1}{1+\mu \varepsilon^{2}}\inf_{\eta, \mathcal{X}}\sup_{q}\left\lbrace u^{\varepsilon}(t-\varepsilon^{2}, p \cdot \delta_\varepsilon(q))+R^{\varepsilon}(t, p, q, \eta, \mathcal{X})\right\rbrace.$$When $k_t=1$, $u^{\varepsilon}(t-\varepsilon^{2}, p\cdot \delta_\varepsilon(q))$ is understood as $u^{\varepsilon}(0, p)$.

\end{definition}

We start with two technical lemmas regarding some Lipschitz regularity of the value functions for smooth initial datum $\psi$. 
\begin{lemma}\label{lemmalipschitzJtime} Let $\psi$ be smooth and so that the right and left horizontal derivatives of first and second order are bounded in $\mathbb{G}$. Then, there is a constant $C=C[\psi]>0$  such that for all $p, \hat{p} \in \mathbb{G}$ and $t \in [0, T]$:
\begin{equation*}
    \vert u^{\varepsilon}(t, p)-u^{\varepsilon}(t, \hat{p}) \vert \leq C \vert p\cdot \hat{p}^{-1} \vert_{\mathbb{G}} + k_t\varepsilon^{2}\omega \left( \vert p\cdot \hat{p}^{-1} \vert_{\mathbb{G}}\right),
\end{equation*}where $\omega$ is the modulus of continuity from $(\mathcal{F} 3)$.
\end{lemma}

\begin{proof}We proceed by induction. When $t=0$ ($k_t=0$), we have by the boundedness of the horizontal derivatives and the  stratified  mean value theorem \cite[Theorem 1.41]{FS} (modified for right-invariant vector fields, see Remark after \cite[Theorem 1.37]{FS}) that there is $C=C[\psi]>0$ so that:
$$ \vert u^{\varepsilon}(0, p)-u^{\varepsilon}(0, \hat{p}) \vert =  
\vert \psi(p ) - \psi( \hat{p} ) \vert \leq C| p\cdot \hat{p}^{-1}|_\mathbb{G}$$

 Assuming for $t \in ((k_t-1)\varepsilon^{2}, k_t\varepsilon^{2}]$, $k_t \in \left\lbrace 1, ..., m-1\right\rbrace$:
\begin{equation*}
    \vert u^{\varepsilon}(t, p)-u^{\varepsilon}(t, \hat{p}) \vert \leq C\vert p\cdot \hat{p}^{-1} \vert_{\mathbb{G}}+ k_t\varepsilon^{2}\omega\left( \vert p\cdot \hat{p}^{-1} \vert_{\mathbb{G}} \right)
\end{equation*}we have for $\tilde{t}=t+\varepsilon^{2}$:
\begin{equation*}
\begin{split}
& u^{\varepsilon}(\tilde{t}-\varepsilon^{2},  p \cdot \delta_\varepsilon(q) )-u^{\varepsilon}(\tilde{t}-\varepsilon^{2}, \hat{p}\cdot \delta_\varepsilon(q)) +R^{\varepsilon}(\tilde{t}, p, q, \eta, \mathcal{X})-R^{\varepsilon}(\tilde{t},  \hat{p}, q, \eta, \mathcal{X}) \\ & \leq C\vert p\cdot \hat{p}^{-1}\vert_{\mathbb{G}} + k_t\varepsilon^{2}\omega\left( \vert p\cdot \hat{p}^{-1} \vert_{\mathbb{G}} \right) + \varepsilon^{2}\omega \left(\vert p\cdot \hat{p}^{-1}\vert_{\mathbb{G}}  \right) \leq C\vert p\cdot \hat{p}^{-1}\vert_{\mathbb{G}} + k_{\tilde{t}}\varepsilon^{2}\omega\left( \vert p\cdot \hat{p}^{-1} \vert_{\mathbb{G}} \right), \quad  k_{\tilde{t}}:=k_t+1.
\end{split}
\end{equation*}Thus:

\begin{equation*}
\begin{split}
& u^{\varepsilon}(\tilde{t}-\varepsilon^{2},  p \cdot \delta_\varepsilon(q) )+R^{\varepsilon}(\tilde{t},  p, q, \eta, \mathcal{X})\\ & \leq  C\vert p\cdot \hat{p}^{-1} \vert_{\mathbb{G}} + k_{\tilde{t}}\varepsilon^{2}\omega\left( \vert p\cdot \hat{p}^{-1} \vert_{\mathbb{G}} \right) + u^{\varepsilon}(\tilde{t}-\varepsilon^{2}, \hat{p} \cdot \delta_\varepsilon(q) ) + R^{\varepsilon}(\tilde{t},  \hat{p}, q,  \eta, \mathcal{X})
\end{split}
\end{equation*}Taking $\sup_q$ and then $\inf_{\eta, \mathcal{X}}$ we derive:

$$|u^{\varepsilon}(\tilde{t}, p)-u^{\varepsilon}(\tilde{t}, \hat{p})| \leq C\vert p\cdot \hat{p}^{-1}  \vert_{\mathbb{G}} + k_{\tilde{t}}\varepsilon^{2}\omega\left( \vert p\cdot \hat{p}^{-1}  \vert_{\mathbb{G}} \right).$$

\end{proof}

\begin{lemma}\label{lemmalipschitzJspace}Let $\psi$ be smooth and  so that the right and left horizontal derivatives of first and second order are bounded in $\mathbb{G}$.   Then, there is $C=C(\psi, \lambda_0, \lambda_1)>0$ such that for all $p \in \overline{\Omega}$ and $t$: 
\begin{equation*}
    \vert u^{\varepsilon}(t, p)-u^{\varepsilon}(t-\varepsilon^{2}, p) \vert \leq C \left( \dfrac{1}{1 + \mu \varepsilon^2} \right)^{k_t} \varepsilon^2+ (k_t-1)\varepsilon^{2}\omega(\varepsilon^{2}),
\end{equation*}where $\omega$ is the modulus of continuity from $(\mathcal{F}3)$.   
\end{lemma}
\begin{proof}
We proceed by induction. Since $\psi$ is smooth, and has bounded  derivatives, we derive from the stratified Taylor formula \eqref{stratified taylos formula} and the fact that $\delta_\varepsilon$ is horizontal that there is $C[\psi]>0$ so that:
\begin{equation}\label{aprox secod order}
\begin{split}
& |\psi(p\cdot \delta_\varepsilon(q))-\psi(p) - \left\langle \nabla_{\mathcal{G}, 0} \psi(p), [\delta_\varepsilon(q)]_1 \right\rangle -\frac{1}{2}\left\langle \nabla^{2, *}_{\mathcal{G}, 0} \psi(p) [\delta_\varepsilon(q)]_1, [\delta_\varepsilon(q)]_1 \right\rangle|\\ & \qquad \qquad  \leq C[\psi]\varepsilon^{3/2}\cdot \sup_{|z|_\mathbb{G}\leq b \varepsilon^{3/4}, d(I)=2}|X^{I}\psi(p\cdot z)-X^{I}\psi(p)| \qquad \text{ for some universal }b>0\\ & \qquad \qquad \leq C[\psi]\varepsilon^{2},
\end{split}
\end{equation}where we have used the stratified mean value theorem \cite[Theorem 1.41]{FS} for the second order derivatives. Hence for $t \in (0, \varepsilon^{2}]$:
\begin{equation*}
    \begin{split}
       u^{\varepsilon}(t, p)-\psi(p)&\leq \left(\frac{1}{1+\mu \varepsilon^{2}}\right)\inf_{\eta, \mathcal{X}} \sup_{q} \left\lbrace  \varepsilon \langle (\nabla_{\mathcal{G}, 0} \psi(p) - \eta), \nu \rangle \right. \\
        &\qquad  + \dfrac{\varepsilon^{2}}{2} \langle \left( \nabla_{\mathcal{G}, 0}^{2, *}\psi(p) - \mathcal{X} \right) \nu, \nu \rangle - \left. \varepsilon^{2} \mathcal{F}(t, p, \eta, \mathcal{X}) +C[\psi]\varepsilon^{2}\right\rbrace + \frac{\mu \varepsilon^{2}}{1+\mu\varepsilon^{2}}\psi(p)\\& \leq \left(\frac{1}{1+\mu \varepsilon^{2}}\right)\inf_{\eta, \mathcal{X}}  \left\lbrace  \varepsilon^{3/4} \Vert\nabla_{\mathcal{G}, 0} \psi(p) - \eta \Vert \right. \\
        &\qquad  + \dfrac{\varepsilon^{3/2}}{2} \mathcal{E}^{+}\left( C_0[\psi]I - \mathcal{X} \right)  + C\left. \varepsilon^{2} (1+\|\mathcal{X}\|)+C[\psi]\varepsilon^{2}\right\rbrace + C_0[\psi]\varepsilon^{2},
    \end{split}
\end{equation*}where:
\begin{equation}\label{eqdefC0}
    C_{0}[\psi] := \max\left[ \Vert \psi \Vert_{\infty}, \Vert \nabla_{\mathcal{G}, 0}\psi\Vert_{\infty}, \left\Vert \nabla_{\mathcal{G}, 0}^{2, *}\psi\right\Vert_{\infty} \right]
\end{equation}and $I\in S^{N}$ denotes the identity matrix. Let $\varepsilon > 0$ be small enough so that $C_{0}[\psi]$ $\leq \varepsilon^{-1/4}$, then we can choose $(\eta, \mathcal{X}) =\left( \nabla_{\mathcal{G}, 0} \psi (p), C_{0}[\psi] I \right)$ (in the case  $\nabla_{\mathcal{G}, 0} \psi (p)$, take an approximating sequence $0 \neq\eta_n \to 0$) to obtain:
\begin{equation}\label{eqboundedbelowJ-psi}
    \begin{split}
        u^{\varepsilon}(t, p)-\psi(p) &\leq \left[ C \left( 1+ C_{0}[\psi] \right)+C[\psi]\right]\left( \dfrac{1}{1+ \mu \varepsilon^{2}} \right) \varepsilon^{2}.
    \end{split}
\end{equation}
Next, we show the lower bound. Similar to the above arguments, for any $q$ with $\vert q \vert_{\mathbb{G}} \leq \varepsilon^{-1/4}$ we have:
\begin{equation*}
    \begin{split}
        (1+\mu\varepsilon^{2})(u^{\varepsilon}(t, p)-\psi(p)) &\geq \inf_{\eta, \mathcal{X}} \left\lbrace  \varepsilon \langle (\nabla_{\mathcal{G}, 0} \psi(p) - \eta), \nu \rangle \right.+  \dfrac{\varepsilon^{2}}{2} \langle \left( -C_{0}[\psi] I - \mathcal{X} \right) \nu, \nu \rangle  - \left. \varepsilon^{2} \mathcal{F}(t, p, \eta, \mathcal{X}) -C[\psi]\varepsilon^{2} \right\rbrace -C_0[\psi]\varepsilon^{2}. 
    \end{split}
\end{equation*}
Applying Lemma \ref{lemmafundamental} from the Appendix with $\hat{\eta}=\nabla_{\mathcal{G}, 0} \psi(p)$, $\hat{\mathcal{X}}=-C_{0}[\psi]I$ and choosing an appropriate
$q=\overline{q}(\epsilon,\eta, \hat{\eta}, \mathcal{X}, \hat{\mathcal{X}})$, we have in the case $\Vert \nabla_{\mathcal{G}, 0} \psi(p) \Vert \geq 1$:
\begin{equation*}
    \begin{split}
         (1+\mu\varepsilon^{2})(u^{\varepsilon}(t, p)-\psi(p))&\geq - \varepsilon^{2} \mathcal{F}^{*}(t, p, \nabla_{\mathcal{G}, 0} \psi(p), -C_0[\psi]I)
        - h_1(\varepsilon^{1/4}) \varepsilon^2-C[\psi]\varepsilon^{2}-C_0[\psi]\varepsilon^{2},
    \end{split}
\end{equation*}and:
\begin{equation*}
    \begin{split}
         (1+\mu\varepsilon^{2})(u^{\varepsilon}(t, p)-\psi(p))&\geq - \varepsilon^{2} \mathcal{F}^{*}(t, p, 0, -C_0[\psi]I)
        -C[\psi]\varepsilon^{2}-C_0[\psi]\varepsilon^{2},
    \end{split}
\end{equation*}
if $\Vert \nabla_{\mathcal{G}, 0} \psi(p) \Vert \leq 1$, for all sufficiently small $\varepsilon \leq \min[\varepsilon_{1}, \varepsilon_{2}]$ where $R_{0} := C_{0}[\psi]$ and $K=1$.  As in \eqref{eqboundedbelowJ-psi}, we have:
\begin{equation}\label{eqboundedlowerJ-psi}
    \begin{split}
        u^{\varepsilon}(t, p)-\psi(p)&\geq - \left[ C \left( 1+ C_{0}[\psi]\right) + h_1(\varepsilon^{1/4}) + C[\psi]\right] \left( \dfrac{1}{1+ \mu \varepsilon^{2}} \right) \varepsilon^{2}.
    \end{split}
\end{equation}
Combining \eqref{eqboundedbelowJ-psi} and \eqref{eqboundedlowerJ-psi},  we obtain:
\begin{equation}\label{eqboundedJ-psi}
    |u^{\varepsilon}(t, p)-\psi(p)| \leq C\left(\frac{1}{1+\mu \varepsilon^{2}}\right)\varepsilon^{2}.
\end{equation}

Suppose now that:
\begin{equation*}
|u^{\varepsilon}(t-\varepsilon^{2}, p)-u^{\varepsilon}(t, p)| \leq C\left(\frac{1}{1+\mu \varepsilon^{2}}\right)^{k_t}\varepsilon^{2}+(k_t-1)\varepsilon^{2}\omega(\varepsilon^{2}).
\end{equation*}Now, taking $\tilde{t}=t+\varepsilon$:
\begin{equation*}
\begin{split}
& u^{\varepsilon}(t, p\cdot \delta_\varepsilon(q))-u^{\varepsilon}(t-\varepsilon^{2}, p\cdot \delta_\varepsilon(q)) +R^{\varepsilon}(\tilde{t}, p, \eta, \mathcal{X})-R^{\varepsilon}(\tilde{t}-\varepsilon^{2}, p, \eta, \mathcal{X})\\ & \quad \leq C\left(\frac{1}{1+\mu \varepsilon^{2}}\right)^{k_t}\varepsilon^{2}+(k_t-1)\varepsilon^{2}\omega(\varepsilon^{2}) + \varepsilon^{2}\omega(\varepsilon^{2}). 
\end{split}
\end{equation*}where we have used $(\mathcal{F}3)$ in the latter inequality. Taking $\sup_q$ and $\inf_{\eta, \mathcal{X}}$, we derive:
\begin{equation*}
u^{\varepsilon}(\tilde{t}-\varepsilon^{2}, p)-u^{\varepsilon}(\tilde{t}, p) \leq C\left(\frac{1}{1+\mu \varepsilon^{2}}\right)^{k_t+1}\varepsilon^{2}+k_t\varepsilon^{2}\omega(\varepsilon^{2})=C\left(\frac{1}{1+\mu \varepsilon^{2}}\right)^{k_{\tilde{t}}}\varepsilon^{2}+(k_{\tilde{t}}-1)\varepsilon^{2}\omega(\varepsilon^{2}).
\end{equation*}A similar argument is applied to $u^{\varepsilon}(\tilde{t}, p)-u^{\varepsilon}(\tilde{t}-\varepsilon^{2}, p)$. 
\end{proof}

In the next results we shall appeal to the following constant. For a positive integer $K$,  let us set:
\begin{equation*}
    C^{\varepsilon}[\psi, K] := C(1 + C_0[\psi] )+ h_{K}(\varepsilon^{1/4}) + C[\psi],
\end{equation*}where $C$ is the constant from Remark \ref{remarkhipotesysofFandH}, $C_0$ is given by \eqref{eqdefC0}, $C[\psi]$ by \eqref{aprox secod order}  and $h_K$ is the modulus in Lemma \ref{lemmafundamental} from the Appendix.

The next proposition established the convergence of the value functions. 
\begin{proposition}There exists a subsequence $\left\lbrace \varepsilon_j \right\rbrace_j$ converging to $0$ and a  continuous function $u$ so that:
$$u^{\varepsilon_j} \to u$$locally uniformly as $j \to \infty.$ Moreover, $u \in BUC([0, T]\times \mathbb{G})$ and $u(0, p)=\psi(p)$ for all $p$.
\end{proposition}
\begin{proof} For  $\delta>0$ consider the  regularizations $\psi_{\delta}^{\pm} \in \mathcal{C}^{2}$ of $\psi$ from assumption (E):
\begin{equation}\label{ineqpsidelta+-}
    \psi - \delta \leq \psi_{\delta}^{-} \leq \psi \leq \psi_{\delta}^{+} \leq \psi + \delta,
\end{equation}Lemma \ref{lemmalipschitzJtime} implies the estimate for $t \in (0, \varepsilon^{2}]$:
\begin{equation}\label{eqJdelta+-bounded}
    \vert u^{\varepsilon}_{\psi_{\delta}^{\pm}} (t, p) -u^{\varepsilon}_{\psi_{\delta}^{\pm}} (t, \hat{p}) \vert \leq C_{\delta} \vert p \cdot  \hat{p}^{-1}\vert_{\mathbb{G}} +\varepsilon^{2}\omega( \vert p \cdot  \hat{p}^{-1}\vert_{\mathbb{G}} ) ,
\end{equation}
for all $p, \hat{p}$ and all sufficiently small $\varepsilon$, and where $u^{\varepsilon}_{\psi_{\delta}^{\pm}} $ denotes the value function  with $\psi=\psi_{\delta}^{\pm}$.  
%\begin{equation*}
%    \begin{split}
%        \mathcal{J}_{\varepsilon} \psi &\leq \mathcal{J}_{\varepsilon} \psi_{\delta}^{+} \leq \mathcal{J}_{\varepsilon} \psi + \left( \dfrac{1}{1+\mu\varepsilon^2} \right) \delta, \\
%        \mathcal{J}_{\varepsilon} \psi - \left( \dfrac{1}{1+\mu\varepsilon^2} \right) \delta &\leq \mathcal{J}_{\varepsilon} \psi_{\delta}^{-} \leq \mathcal{J}_{\varepsilon} \psi.
%    \end{split}
%\end{equation*}
%Hence:
%\begin{equation}\label{ineqpsidelta+-2}
%    \mathcal{J}_{\varepsilon} \psi - \delta \leq \mathcal{J}_{\varepsilon} \psi_{\delta}^- \leq \mathcal{J}_{\varepsilon} \psi \leq \mathcal{J}_{\varepsilon} \psi_{\delta}^+ \leq \mathcal{J}_{\varepsilon} \psi + \delta.
%\end{equation}
Combining \eqref{eqJdelta+-bounded} and \eqref{ineqpsidelta+-}, we conclude that:
\begin{equation*}
    \vert u^{\varepsilon}(t, p)-u^{\varepsilon}(t, \hat{p}) \vert \leq C_{\delta} \vert p \cdot  \hat{p}^{-1}\vert_{\mathbb{G}} +\varepsilon^{2}\omega( \vert p \cdot  \hat{p}^{-1}\vert_{\mathbb{G}} ) + \delta,
\end{equation*}
for $p, \hat{p} \in \overline{\Omega}$ and all $\varepsilon \leq \varepsilon'$. Here $\varepsilon' = \varepsilon' (\psi_{\delta}^{\pm}, \lambda_{0}, \lambda_{1})$ is sufficiently small. Inductively, we derive for $t \in ((k_t-1)\varepsilon^{2}, k_t \varepsilon^{2}]$:
\begin{equation}\label{ineqlemmalipschitzspacenotC2}
    \vert u^{\varepsilon}(t, p)-u^{\varepsilon}(t, \hat{p}) \vert \leq C_{\delta} \vert p \cdot  \hat{p}^{-1}\vert_{\mathbb{G}} +k_t\varepsilon^{2}\omega( \vert p \cdot  \hat{p}^{-1}\vert_{\mathbb{G}} ) + \delta.
\end{equation}
Now, by Lemma \ref{lemmalipschitzJspace}, the estimate:
\begin{equation*}
    \vert u^{\varepsilon}_{\psi_{\delta}^{\pm}} (t, p) -u^{\varepsilon}_{\psi_{\delta}^{\pm}} (t-\varepsilon^{2}, p) \vert \leq C^{\varepsilon} [\psi_{\delta}^{\pm}] \left( \dfrac{1}{1+\mu\varepsilon^2} \right)^{k_t} \varepsilon^2 + (k_t-1)\varepsilon^{2}\omega(\varepsilon^{2}),
\end{equation*} 
holds for each $t$ and all sufficiently small $\varepsilon$  where $C^{\varepsilon} [\psi_{\delta}^{\pm}] :=  \max \left[ C^{\varepsilon} [\psi_{\delta}^+, 1], C^{\varepsilon} [\psi_{\delta}^-, 1] \right]$. If $0 \leq i \leq j \leq m$, and $t \in ((i-1)\varepsilon^{2}, i\varepsilon^{2}]$, $s \in ((j-1)\varepsilon^{2}, j\varepsilon^{2}]$,  then we have:
\begin{equation}
\begin{split}
& u^{\varepsilon}_{\psi_{\delta}^{\pm}} (t-\varepsilon^{2}, p \cdot \delta_\varepsilon(q)) -u^{\varepsilon}_{\psi_{\delta}^{\pm}} (s-\varepsilon^{2}, p \cdot \delta_\varepsilon(q)) +R^{\varepsilon}(t, p, q, \eta, \mathcal{X})-R^{\varepsilon}(s, p, q, \eta, \mathcal{X}) \\ & \,\, \leq u^{\varepsilon}_{\psi_{\delta}^{\pm}} (t-\varepsilon^{2}, p \cdot \delta_\varepsilon(q)) \pm u^{\varepsilon}_{\psi_{\delta}^{\pm}} (t-\varepsilon^{2}+\varepsilon^{2}, p \cdot \delta_\varepsilon(q))\pm \cdots \pm u^{\varepsilon}_{\psi_{\delta}^{\pm}} (t-\varepsilon^{2}+(j-i)\varepsilon^{2}, p \cdot \delta_\varepsilon(q))  -u^{\varepsilon}_{\psi_{\delta}^{\pm}} (s-\varepsilon^{2}, p \cdot \delta_\varepsilon(q)) \\ & \quad+ \varepsilon^{2}\omega(|s-t|) \\ &  \, \, \leq C^{\varepsilon} [\psi_{\delta}^{\pm}](j-i)\varepsilon^{2}+ j(j-i)\varepsilon^{2}w(\varepsilon^{2})+ \varepsilon^{2}\omega(|s-t|) 
\end{split}
\end{equation}Hence, taking $\sup_q$ and $\inf_{\eta, \mathcal{X}}$ we derive for $t \in ((i-1)\varepsilon^{2}, i\varepsilon^{2}]$, $s \in ((j-1)\varepsilon^{2}, j\varepsilon^{2}]$:
$$|u^{\varepsilon}_{\psi_{\delta}^{\pm}}  (t, p) -u^{\varepsilon}_{\psi_{\delta}^{\pm}}  (s, p)| \leq C^{\varepsilon} [\psi_{\delta}^{\pm}](j-i)\varepsilon^{2}+ j(j-i)\varepsilon^{2}w(\varepsilon^{2})+ \varepsilon^{2}\omega(|s-t|). $$

Thus:

\begin{equation}\label{ineqlemmalipschitztimenotC2}
    |u^{\varepsilon} (t, p) -u^{\varepsilon}  (s, p)| \leq C^{\varepsilon} [\psi_{\delta}^{\pm}](j-i)\varepsilon^{2}+ m(j-i)\varepsilon^{2}\omega(\varepsilon^{2})+ \varepsilon^{2}\omega(|s-t|)+ \delta.
\end{equation}Now we prove the proposition. For any $t, s \in [0, T]$ with $t \leq s$, there exist $i, j$ such that $0 \leq i \leq j \leq m$ and:
\begin{equation*}
    j \varepsilon^2 \leq s < (j+1) \varepsilon^2,\,\, i \varepsilon^2 \leq t < (i+1)\varepsilon^{2}.
\end{equation*} From \eqref{ineqlemmalipschitzspacenotC2} and \eqref{ineqlemmalipschitztimenotC2}, it follows:
\begin{equation*}
    \vert u^{\varepsilon}(t, p) - u^{\varepsilon}(s, \hat{p}) \vert \leq C^{\varepsilon} [\psi_{\delta}^{\pm}] (j-i) \varepsilon^2 + (j-i)T\omega(\varepsilon^{2})+ \varepsilon^{2}\omega(|s-t|)+  C_{\delta} \vert p \cdot  \hat{p}^{-1}\vert_{\mathbb{G}} +T\omega( \vert p \cdot  \hat{p}^{-1}\vert_{\mathbb{G}} ) + 2\delta.
\end{equation*}
Set $C_{0}[\psi_{\delta}^{\pm}] := \max \left[ C_{0} [\psi_{\delta}^+] , C_{0} [\psi_{\delta}^-] \right]$. Since $i \varepsilon^2 > t - \varepsilon^{2}$ and $j \varepsilon^2 \leq s$, we have:
\begin{equation*}
   \vert u^{\varepsilon}(t, p) - u^{\varepsilon}(s, \hat{p}) \vert \leq C^{\varepsilon} [\psi_{\delta}^{\pm}] (s-t) +m (s-t)\omega(\varepsilon^{2})+ \varepsilon^{2}\omega(|s-t|)+  C_{\delta} \vert p \cdot  \hat{p}^{-1}\vert_{\mathbb{G}} +T\omega( \vert p \cdot  \hat{p}^{-1}\vert_{\mathbb{G}} ) + 2\delta.
\end{equation*}
%for all sufficiently small $\varepsilon$ so that $\varepsilon \leq \varepsilon'$ where $\varepsilon' = \varepsilon'(R_{0}, n_1, \lambda_0, \lambda_1)$ with $R_{0} := C_{0} [\psi_{\delta}^{\pm}]$. 
Interchanging  $s$ and $t$, we conclude:
\begin{equation}\label{ineqlemmatimespaceforu}
   \vert u^{\varepsilon}(t, p) - u^{\varepsilon}(s, \hat{p}) \vert \leq C^{\varepsilon} [\psi_{\delta}^{\pm}] |s-t| + m |s-t|\omega(\varepsilon^{2})+ \varepsilon^{2}\omega(|s-t|)+  C_{\delta} \vert p \cdot  \hat{p}^{-1}\vert_{\mathbb{G}} +T\omega( \vert p \cdot  \hat{p}^{-1}\vert_{\mathbb{G}} ) + 2\delta.
\end{equation}for all $s, t \in [0, T]$. Making use of the assumption $(\mathcal{F}3)$ on $\omega$ there is a constant $C$ so that:
$$m\omega(\varepsilon^{2}) \leq C.$$ 
 Also, observe that:
\begin{equation*}
    C[\psi_{\delta}^{\pm}] := \lim_{\varepsilon \to 0} C^{\varepsilon} [\psi_{\delta}^{\pm}] = C (1 + C_0 [\psi_{\delta}^{\pm}]).
\end{equation*}Next, take $\eta \in (0, 1)$ and fix $\delta < \eta/6$. Moreover, consider $ \varepsilon_0>0$ so that $\varepsilon_0< \eta/4$ and  for all $\varepsilon < \varepsilon_0$:
$$C^{\varepsilon}[\psi_\delta^{\pm}] \leq C(1+C_0[\psi_\delta^{\pm}]) + \delta.$$Finally, take $r_0>0$ so that $\vert \hat{p}^{-1} \cdot p \vert_{\mathbb{G}} + \vert s - t \vert  < r_0$ gives:
$$\left[\frac{\eta}{4}+ C(1+C_0[\psi_\delta^{\pm}])  \right] r_0 + \varepsilon_0^{2}\omega(r_0)+  C_{\delta} r_0+T\omega(r_0)  < \frac{\eta}{2}.$$Therefore, $\vert \hat{p}^{-1} \cdot p \vert_{\mathbb{G}} + \vert s - t \vert  < r_0$ and $\varepsilon < \varepsilon_0$ imply:
$$|u^{\varepsilon}(t, p)-u^{\varepsilon}(s, \hat{p})| < \eta.$$

Moreover, the functions $u^{\varepsilon}$ are uniformly bounded. Indeed, we proceed to prove the claim by induction. For $t \in (0, \varepsilon^{2}]$,  we have the upper bound:
\begin{equation*}
    \begin{split}
      (1+\mu \varepsilon^{2}) u^{\varepsilon}(t, p) &\leq \Vert \psi \Vert_{\infty} +\ \inf_{\eta, \mathcal{X}} \sup_{q} R^{\varepsilon}(t, p, q, \eta, \mathcal{X}) \\
        &\leq  \Vert \psi \Vert_{\infty} + \inf_\eta \sup_q  R^{\varepsilon}(t, p, q,  \eta, \mathcal{O})  \\ & =  \Vert \psi \Vert_{\infty} + \inf_\eta \sup_q  \left(-\varepsilon \left\langle \eta, \nu \right\rangle -\varepsilon^{2}\mathcal{F}_*(t, p, \eta, \mathcal{O})\right).
        \end{split}
\end{equation*}Taking a sequence  $\eta_k \searrow 0$ with $\|\eta_k\|\leq \varepsilon^{-1/4}$ and using the lower semicontinuity of $\mathcal{F}_*$ together with $(\mathcal{F}5)$, we derive:
\begin{equation*}
    \begin{split}
         u^{\varepsilon}(t, p)  \leq \Vert \psi \Vert_{\infty}
    \end{split}
\end{equation*}
for all $\varepsilon > 0$. Now, taking $\hat{\eta}=0,$ $\hat{\mathcal{X}}=O$ and $R_{0}=1$ in Lemma \ref{lemmafundamental}, there is $\overline{q}$ (depending on $\varepsilon$, $\eta$ and $\mathcal{X}$) so that:
\begin{equation*}
    \begin{split}
         (1+\mu \varepsilon^{2}) u^{\varepsilon}(t, p) &\geq -\Vert \psi \Vert_{\infty} + \inf_{\eta, \mathcal{X}} \sup_{q} R^{\varepsilon}(t, p , q, \eta, \mathcal{X}) \\
        &\geq -\Vert \psi \Vert_{\infty}+\inf_{\eta, \mathcal{X}}   R^{*, \varepsilon}(t, p, \overline{q}, 0, \mathcal{O}) \\
        &= -\Vert \psi \Vert_{\infty}. 
    \end{split}
\end{equation*}Thus:
\begin{equation*}
    \vert u^{\varepsilon}(t, p)\vert \leq \Vert \psi \Vert_{\infty}.
\end{equation*}
By induction we deduce for all $(t, p)$:
\begin{equation*}
    \vert u^{\varepsilon} (t, p) \vert \leq \Vert \psi \Vert_{\infty}.
\end{equation*}

In this way, we may apply Lemma 4.2 in \cite{MPR} to get the convergence (up to a subsequence) of $u^{\varepsilon}$ to some continuous $u$, locally uniformly in $[0, T]\times \mathbb{G}$.

We now prove the final statement.  Taking $\varepsilon_j \to 0$ in \eqref{ineqlemmatimespaceforu}, we derive:
$$|u(t, p)-u(s, \hat{p})| \leq C(1+C_0[\psi_\delta^{\pm}]) |s-t|+C_\delta|p \cdot \hat{p}^{-1}|_\mathbb{G} + T\omega(|p \cdot \hat{p}^{-1}|_\mathbb{G})+ 2 \delta.$$Hence $u \in BUC([0, T]\times \mathbb{G})$.  Applying \eqref{ineqlemmatimespaceforu} to $s=0$ and $p=\hat{p}$, and taking the limit $\varepsilon_j \to 0$  we obtain:
$$|u(t, p)-\psi(p)| \leq C(1+C_0[\psi_\delta^{\pm}]) |t|+ 2 \delta.$$Letting $t \to 0$ and then $\delta \to 0$ it follows $u(0, p)=\psi(p)$.

\end{proof}

\begin{proposition}\label{overline is subsolution} The function $u$ is a viscosity subsolution of \eqref{problem}. 
\end{proposition}

\begin{proof} We argue by contradiction. Then there exist a positive constant $\theta_0$ and a smooth function $\varphi$, such that the following holds in a neighbourhood $\overline{\mathcal{B}}_0 := [t_0 - \delta, t_0 + \delta] \times \overline{B}_{\mathbb{G}}(p_0, r_0)$ of  a strict local maximal point $P_0 = (t_0, p_0) \in (0, T) \times \mathbb{G}$ of $u - \varphi$:
\begin{equation}\label{subsolutiontheta}
    \partial_t \varphi + \mu u+ \mathcal{F}_*(t, p, \nabla_{\mathcal{G}, 0} \varphi, \nabla^{2, *}_{\mathcal{G}, 0} \varphi) \geq \theta_0 > 0,
\end{equation}where $\delta$ and $r_0$ are sufficiently small, with:
\begin{equation}\label{choice deltaa}
4\delta < T.
\end{equation}We also assume: 
\begin{equation*}
    \max\limits_{\overline{\mathcal{B}}_0} (u - \varphi) = 0.
\end{equation*}

Let $P = (t, p) \in \overline{\mathcal{B}}_0$. Reasoning as in \eqref{aprox secod order},  we have:
\begin{equation*}
    \begin{split}
        u^{\varepsilon}(P) &= \left(\dfrac{1}{1 + \mu \varepsilon^2} \right)\inf_{\eta, \mathcal{X}} \sup_{q} \left\lbrace (u^{\varepsilon} - \varphi)(t- \varepsilon^2, p \cdot \delta_{\varepsilon}(q)) + \varphi(P) - \varepsilon^2 \partial_t \varphi(P) \right. \\
        &\qquad + \varepsilon \langle \nabla_{\mathcal{G},  0} \varphi(P) - \eta, \nu \rangle + \dfrac{\varepsilon^2}{2} \left\langle \left( \nabla^{2, *}_{\mathcal{G}, 0} \varphi(P) - \mathcal{X} \right) \nu, \nu \right\rangle  \left. - \varepsilon^2 \mathcal{F} \left(t, p, \eta, \mathcal{X} \right)  + o(\varepsilon^2) \right\rbrace.
    \end{split}
\end{equation*}In the sequel,  $\varepsilon$ is small enough so that: 
\begin{equation*}
    \Vert \nabla_{\mathcal{G},  0} \varphi \Vert_{\infty, \overline{\mathcal{B}}_0} \leq \varepsilon^{-1/4}, \,  \Vert \nabla^{2, *}_{\mathcal{G}, 0} \varphi \Vert_{\infty, \overline{\mathcal{B}}_0} \leq \varepsilon^{-1/2} \textnormal{ and }  o(\varepsilon^2) - \varepsilon^2 \theta_0 \leq 0.
\end{equation*}
Using $-\varphi \leq -u$ in $\overline{\mathcal{B}}_0$ we derive:
\begin{equation*}
    \begin{split}
        (u^{\varepsilon} - \varphi)(P) &\leq \left(\dfrac{1}{1 + \mu \varepsilon^2}\right) \inf_{\eta, \mathcal{X}} \sup_{q} \left\lbrace (u^{\varepsilon} - \varphi)(t- \varepsilon^2, p \cdot \delta_{\varepsilon}(q)) \right. \\
        &\qquad- \varepsilon^2 \left[ \partial_t \varphi(P) + \mu u(P) + \mathcal{F}_* \left(t, p,  \eta, \mathcal{X} \right) \right] \\
        &\qquad+ \left. \varepsilon \langle \nabla_{\mathcal{G}, 0} \varphi(P) - \eta, \nu \rangle + \dfrac{\varepsilon^2}{2} \left\langle \left( \nabla^{2, *}_{\mathcal{G}, 0} \varphi(P) - \mathcal{X} \right) \nu, \nu \right\rangle + o(\varepsilon^2) \right\rbrace.
    \end{split}
\end{equation*}Taking the special choices $\eta = \nabla_{\mathcal{G}, 0} \varphi(P)$, $\mathcal{X} = \nabla^{2, *}_{\mathcal{G}, 0} \varphi(P)$ and appealing to \eqref{subsolutiontheta} we deduce: 
\begin{equation*}
    \begin{split}
        (u^{\varepsilon} - \varphi)(P) &\leq \left(\dfrac{1}{1 + \mu \varepsilon^2}\right) \sup_{q} \left\lbrace (u^{\varepsilon} - \varphi)(t- \varepsilon^2, p \cdot \delta_{\varepsilon}(q)) +  o(\varepsilon^2) - \varepsilon^2 \theta_0 \right\rbrace \\
        &\leq \left(\dfrac{1}{1 + \mu \varepsilon^2}\right) \sup_{q} \left\lbrace (u^{\varepsilon} - \varphi)(t- \varepsilon^2, p \cdot \delta_{\varepsilon}(q)) \right\rbrace \\ & \leq \left(\dfrac{1}{1 + \mu \varepsilon^2}\right) \sup_{q} \left\lbrace ((u^{\varepsilon})^{*} - \varphi)(t- \varepsilon^2, p \cdot \delta_{\varepsilon}(q)) \right\rbrace ,
    \end{split}
\end{equation*}Taking a sequence of points $P_n=(t_n, p_n) \in \overline{\mathcal{B}}_0$ converging to $P$ so that:
$$(u^{\epsilon})^{*}(P)=\lim_{n\to \infty}u^{\varepsilon}(P_n)$$we derive:
\begin{equation}\label{ineqsubsolution}
\begin{split}
  ((u^{\varepsilon})^{*} - \varphi)(P) &=\lim_{n\to \infty}(u^{\varepsilon}-\varphi)(P_n) \\ &\leq \left(\dfrac{1}{1 + \mu \varepsilon^2}\right)\lim_{n\to \infty}\sup_{q} \left\lbrace ((u^{\varepsilon})^{*} - \varphi)(t_n- \varepsilon^2, p_n \cdot \delta_{\varepsilon}(q)) \right\rbrace \\ & = \left(\dfrac{1}{1 + \mu \varepsilon^2}\right)\lim_{n\to \infty} \left\lbrace ((u^{\varepsilon})^{*} - \varphi)(t_n- \varepsilon^2, p_n \cdot \delta_{\varepsilon}(q_n)) \right\rbrace \qquad \text{for some }|q_n|_\mathbb{G}\leq \varepsilon^{-1/4} \\ & \leq  \left(\dfrac{1}{1 + \mu \varepsilon^2}\right) ((u^{\varepsilon})^{*} - \varphi)(t- \varepsilon^2, p \cdot \delta_{\varepsilon}(q^{\varepsilon}_0))
  \end{split}
\end{equation}where $q_n \to q^{\varepsilon}_0$ (up to a subsequence that we do not re label) and where we have used the upper semicontinuity of $(u^{\varepsilon})^{*} - \varphi$.

Define $P_0^{\varepsilon}=P_0$ and, for $k\geq 1$,  $P_k^{\varepsilon}= (t_k^{\varepsilon}, p_k^{\varepsilon})$ as follows:
\begin{equation*}
    P_k^{\varepsilon} =  (t_{k-1}^{\varepsilon}-\varepsilon^{2},p_{k-1}^{\varepsilon}\cdot \delta_\varepsilon(q_0^{\varepsilon}(P_{k-1}^{\varepsilon})) ) , \qquad 1 \leq k \leq m.
\end{equation*}If $P_1^{\varepsilon}, P_2^{\varepsilon}, \ldots, P_k^{\varepsilon} \in \overline{\mathcal{B}}_0$, from \eqref{ineqsubsolution}, we obtain:
\begin{equation*}
        ((u^{\varepsilon})^{*} - \varphi)(P_{k - 1}^{\varepsilon}) \leq \left(\dfrac{1}{1 + \mu \varepsilon^2}\right) ((u^{\varepsilon})^* - \varphi)(P_k^{\varepsilon} ),
\end{equation*}
and so:
\begin{equation}\label{ineqsubsolution2}
        ((u^{\varepsilon})^{*} - \varphi)(P_0^{\varepsilon}) \leq \left( \dfrac{1}{1 + \mu \varepsilon^2} \right)^k ((u^{\varepsilon})^* - \varphi)(P_k^{\varepsilon}).
\end{equation}
 Taking $n = n^\varepsilon$ so that $n\varepsilon^{2} \in (\delta, 4\delta)$ it follows $t_n^{\varepsilon} \notin  [t_0 - \delta, t_0 + \delta]$ (i.e., $P_n^{\varepsilon} \notin \overline{\mathcal{B}}_0$).  In addition, $n \leq m$, by the choice \eqref{choice deltaa}. Hence, there exists a minimal number $K_\varepsilon\leq m$ such that $P_{K_\varepsilon}^{\varepsilon} \in \overline{\mathcal{B}}_0$ but $P_{K_\varepsilon+1}^{\varepsilon} \notin \overline{\mathcal{B}}_0$. By compactness,  $P_{K_\varepsilon}^{\varepsilon} \to P'=(t', p') \in \overline{\mathcal{B}}_0 \backslash \lbrace P_0 \rbrace$ as $\varepsilon \to 0$ (or equivalently $m \to \infty$).  Note that:

\begin{equation*}
    \begin{split}
        0 < e^{-\mu T} = e^{-\mu m \varepsilon^2} &\leq \left( 1 + \mu \varepsilon^2 \right)^{-m} \leq \left( 1 + \mu \varepsilon^2 \right)^{-K_\varepsilon} \leq 1.
    \end{split}
\end{equation*}Thus: 
\begin{equation*}
    \lim\limits_{\substack{\varepsilon \to 0 \\ (m \to \infty)}} \left( \dfrac{1}{1 + \mu \varepsilon^2} \right)^{K_{\varepsilon}} := \alpha \in (0, 1].
\end{equation*}
Consequently 
\begin{equation*}
    \begin{split}
        0 = (u - \varphi)(P_0) &= \lim\limits_{\varepsilon \to 0} (u^{\varepsilon} - \varphi)(P_0^{\varepsilon}) \\
        &\leq \lim\limits_{\varepsilon \to 0} \left( \dfrac{1}{1 + \mu \varepsilon^2} \right)^{K_\varepsilon} ((u^{\varepsilon})^* - \varphi)(P_{K_\varepsilon}^{\varepsilon}) \\
        &\leq \alpha (u - \varphi)(P'), \textnormal{ with } P' \in \overline{\mathcal{B}}_0 \backslash \lbrace P_0 \rbrace.
    \end{split}
\end{equation*}
Therefore we get a contradiction, since $P_0$ is a strict maximum in $\overline{\mathcal{B}}_0$.
\end{proof}

\begin{proposition} The function $u$ is a viscosity supersolution of \eqref{problem}. 
\end{proposition}

\begin{proof}  Reasoning by contradiction again,  there exist  $\theta_0> 0$ and a smooth function $\varphi$, such that the following holds in a neighbourhood $\overline{\mathcal{B}}_0$ of a strict local minimal point $P_0 = (t_0, p_0) \in (0, T) \times \Omega$ of $u - \varphi$:
\begin{equation}\label{supersolutiontheta}
    \partial_t \varphi + \mu u + \mathcal{F}^*(t, p, \nabla_{\mathcal{G}, 0} \varphi, \nabla^{2, *}_{\mathcal{G}, 0} \varphi)  \leq - \theta_0 < 0.
\end{equation}We assume that the value of $u - \varphi$ at $P_0$ is $0$.

Let $R_0 > 0$ be so that:
\begin{equation*}
    \Vert \nabla_{\mathcal{G}, 0} \varphi \Vert_{\infty, \overline{\mathcal{B}}_0}, \Vert \nabla^{2, *}_{\mathcal{G}, 0} \varphi \Vert_{\infty, \overline{\mathcal{B}}_0} \leq R_0.
\end{equation*}
Suppose first that $\nabla_{\mathcal{G}, 0} \varphi(P_0) \neq 0$. We may assume that there is $\gamma_0 > 0$ such that: 
\begin{equation*}
    \Vert \nabla_{\mathcal{G}, 0} \varphi \Vert_{\infty, \overline{\mathcal{B}}_0} \geq \gamma_0 > 0
\end{equation*}Hence there exists  $j_0 \in \mathbb{N}$ such that $\gamma_0 > 1/j_0$. By Lemma \ref{lemmafundamental}  and performing a Taylor expansion as in \eqref{aprox secod order}, we have for $P=(t, p)\in \overline{\mathcal{B}}_0$:
\begin{equation*}
    \begin{split}
        u^{\varepsilon}(P) &\geq \dfrac{1}{1 + \mu \varepsilon^2} \inf_{\eta, \mathcal{X}} \left\lbrace (u^{\varepsilon} - \varphi)((t- \varepsilon^2, p \cdot \delta_{\varepsilon}(\overline{q})) + \varphi(P) - \varepsilon^2 \partial_t \varphi(P) \right.  \\
        &\qquad  - \varepsilon^2 \mathcal{F}^* \left(t, p, \nabla_{\mathcal{G}, 0} \varphi(P), \nabla^{2, *}_{\mathcal{G}, 0} \varphi(P)\right) \left. - \varepsilon^2 h_{j_0}(\varepsilon^{1/4}) + o(\varepsilon^2) \right\rbrace,
        \end{split}
\end{equation*}for some $\overline{q} \in \mathbb{G}$, with $\vert \overline{q} \vert_{\mathbb{G}} \leq \varepsilon^{-1/4}$. So we obtain:
\begin{equation}\label{long calc12}
    \begin{split}
        (u^{\varepsilon} -\varphi)(P) &\geq \dfrac{1}{1 + \mu \varepsilon^2} \inf_{\eta, \mathcal{X}} \left\lbrace (u^{\varepsilon} - \varphi)(t- \varepsilon^2, p \cdot \delta_{\varepsilon}(\overline{q})) \right. - \varepsilon^2 \left[ \partial_t \varphi(P) + \mu u(P) \right. \\ 
        &\qquad \left. + \mathcal{F}^* \left(t, p, \nabla_{\mathcal{G}, 0} \varphi(P), \nabla^{2, *}_{\mathcal{G}, 0} \varphi(P)\right) \right] \left. - \varepsilon^2 h_{j_0}(\varepsilon^{1/4}) + o(\varepsilon^2) \right\rbrace \\
        & \geq \dfrac{1}{1 + \mu \varepsilon^2} \inf_{\eta, \mathcal{X}} \left\lbrace (u^{\varepsilon} - \varphi)(t- \varepsilon^2, p \cdot \delta_{\varepsilon}(\overline{q})) + \varepsilon^2 \left[ \theta_0 + o(1) - h_{j_0}(\varepsilon^{1/4}) \right] \right\rbrace \\ & \geq \dfrac{1}{1 + \mu \varepsilon^2} \inf_{\eta, \mathcal{X}} \left\lbrace (u^{\varepsilon} - \varphi)(t- \varepsilon^2, p \cdot \delta_{\varepsilon}(\overline{q}))\right\rbrace \\ & \geq \dfrac{1}{1 + \mu \varepsilon^2} \inf_{\eta, \mathcal{X}} \left\lbrace ((u^{\varepsilon})_{*} - \varphi)(t- \varepsilon^2, p \cdot \delta_{\varepsilon}(\overline{q}))\right\rbrace,
        \end{split}
\end{equation}
where we used $-\varphi \geq -u$ in $\overline{\mathcal{B}}_0$ and \eqref{supersolutiontheta}. By a similar argument as in \eqref{ineqsubsolution}, we derive from \eqref{long calc12} that:
\begin{equation}\label{long calc}
((u^{\varepsilon})_{*} -\varphi)(P) \geq \dfrac{1}{1 + \mu \varepsilon^2} \inf_{\eta, \mathcal{X}} \left\lbrace ((u^{\varepsilon})_{*} - \varphi)(t- \varepsilon^2, p \cdot \delta_{\varepsilon}(\overline{q}))\right\rbrace.
\end{equation}

 Observe that since for each admissible $\eta$ and $\mathcal{X}$ we have  $|\overline{q}|_\mathbb{G} \leq \varepsilon^{-1/4}$, and the function $((u^{\varepsilon})_{*}-\varphi)(t- \varepsilon^2, p \cdot \delta_{\varepsilon}(\cdot))$ is lower semicontinuous, the infimum:
$$\inf_{|q|_\mathbb{G}\leq \varepsilon^{-1/4}} \left\lbrace ((u^{\varepsilon})_{*} - \varphi)(t- \varepsilon^2, p \cdot \delta_{\varepsilon}(\overline{q}))\right\rbrace$$is finite. Hence, so is:
$$\inf_{\eta, \mathcal{X}} \left\lbrace ((u^{\varepsilon})_{*} - \varphi)(t- \varepsilon^2, p \cdot \delta_{\varepsilon}(\overline{q}))\right\rbrace.$$Take a sequence $(\eta_n, \mathcal{X}_n)$  so that if $\overline{q}_n=\overline{q}_n(\varepsilon, \varphi,  P, \eta_n, \mathcal{X}_n)$, then:
$$\lim_{n \to \infty}((u^{\varepsilon})_{*}-\varphi)(t- \varepsilon^2, p \cdot \delta_{\varepsilon}(\overline{q}_n))=\inf_{\eta, \mathcal{X}} \left\lbrace ((u^{\varepsilon})_{*} - \varphi)(t- \varepsilon^2, p \cdot \delta_{\varepsilon}(\overline{q}))\right\rbrace.$$By compactness, there is a point $\overline{q}_0^{\varepsilon}=\overline{q}_0^{\varepsilon}(P)$, $\vert \overline{q}_0^{\varepsilon}\vert_\mathbb{G} \leq \varepsilon^{-1/4}$, so that:

$$  \overline{q}_0^{\varepsilon}= \lim\limits_{n \to \infty}\overline{q}_n.$$The lower semicontinuity of $((u^{\varepsilon})_* - \varphi)(t-\varepsilon^{2}, p \cdot \delta_\varepsilon(\cdot))$ yields:
$$\lim_{n \to \infty}(u^{\varepsilon})_{*}-\varphi)(t- \varepsilon^2, p \cdot \delta_{\varepsilon}(\overline{q}_n)) \geq ((u^{\varepsilon})_{*}-\varphi)(t- \varepsilon^2, p \cdot \delta_{\varepsilon}(\overline{q}^{\varepsilon}_0)).$$Thus, by \eqref{long calc}, we derive:
\begin{equation}\label{new eq}
  ((u^{\varepsilon})_{*} -\varphi)(P) \geq  \left(  \dfrac{1}{1 + \mu \varepsilon^2}\right)((u^{\varepsilon})_{*}-\varphi)(t- \varepsilon^2, p \cdot \delta_{\varepsilon}(\overline{q}^{\varepsilon}_0)).
  \end{equation}

Next, we consider the case $\nabla_{\mathcal{G}, 0} \varphi(P_0) = 0$. Let $\mathbf{F} : \overline{\mathcal{B}}_0 \rightarrow \mathbb{R}$ be so that:
\begin{equation*}
    \mathbf{F}(\cdot):= \partial_t \varphi(\cdot) + \mu u(\cdot) + \mathcal{F}^*(\cdot, 0, \nabla^{2, *}_{\mathcal{G}, 0}\varphi(\cdot)).
\end{equation*}We can assume $\mathbf{F}(P) \leq -\theta_0$ and that $\|\nabla_{\mathcal{G}, 0}\varphi(P)\| \leq 1/j $ for any $P \in \overline{\mathcal{B}}_0$ and some positive integer $j$. 
Applying Lemma \ref{lemmafundamental}, for any $P \in \overline{\mathcal{B}}_0$ and any $\eta, \mathcal{X}$ admissible, there exists $\overline{q}_0^{\varepsilon}$, with $\vert \overline{q}_0^{\varepsilon} \vert_{\mathbb{G}} \leq \varepsilon^{-1/4}$, such that:

    \begin{equation*}
        \begin{split}
            (u^{\varepsilon} - \varphi)(P) &\geq \dfrac{1}{1 + \mu \varepsilon^2} \inf_{\eta, \mathcal{X}} \left\lbrace (u^{\varepsilon} - \varphi)(t-\varepsilon^{2}, p\cdot \delta_\varepsilon(\overline{q}_0^{\varepsilon})) - \varepsilon^2 \left[ \partial_t \varphi(P) + \mu u(P) \right. \right. \\
            &\qquad \left. +\mathcal{F}^* \left(t, p, 0, \nabla^{2, *}_{\mathcal{G},  0} \varphi(P) \right) \right] \left. + o(\varepsilon^2) \right\rbrace \\& \geq \dfrac{1}{1 + \mu \varepsilon^2} \inf_{|q|_\mathbb{G}\leq \varepsilon^{-1/4}} \left\lbrace ((u^{\varepsilon})_* - \varphi)(t-\varepsilon^{2}, p\cdot \delta_\varepsilon(q)) - \varepsilon^2 \textbf{F}(P) + o(\varepsilon^2) \right\rbrace \\ & \geq \dfrac{1}{1 + \mu \varepsilon^2} \inf_{|q|_\mathbb{G}\leq \varepsilon^{-1/4}} \left\lbrace ((u^{\varepsilon})_* - \varphi)(t-\varepsilon^{2}, p\cdot \delta_\varepsilon(q)) \right\rbrace
        \end{split}
    \end{equation*}Thus, there is $q_0^{\varepsilon}=q_0^{\varepsilon}(P)$ where the latter infimum is attained.  Hence  \eqref{new eq} hold.

We may proceed as in the end of Proposition \ref{overline is subsolution} (right after \eqref{ineqsubsolution}) to get a contradiction with the fact that $P_0$ is a strict minimum. 
\end{proof}

\section{Appendix: A  technical lemma for existence}

 We provide the proof of the next lemma which is  \cite[Lemma 4.6]{Kasai}. We give full details to  show  the key point  that $\hat{q}$ may be taken horizontal.

\begin{lemma}\label{lemmafundamental}
    Let $(\hat{\eta}, \hat{\mathcal{X}}) \in \mathbb{R}^{m_1} \times \mathcal{S}^{m_1}\left( \mathbb{R} \right)$ and let $R_0$ so that $\Vert \hat{\eta} \Vert, \Vert \hat{\mathcal{X}} \Vert \leq R_0$.
    \begin{enumerate}
    \item If $\Vert \hat{\eta} \Vert \geq K^{-1}$ ($K \in \mathbb{N}$), then there exists $\varepsilon_1 = \varepsilon_1(K, R_0, \lambda_0, \lambda_1)$ such that for all $(\eta, \mathcal{X}) \in \left( \mathbb{R}^{m_1} \backslash \lbrace 0 \rbrace \right) \times \mathcal{S}^{m_1}\left( \mathbb{R} \right)$, with $\Vert \eta \Vert \leq \varepsilon^{-1/4}, \Vert \mathcal{X} \Vert \leq \varepsilon^{-1/2}$, there exists $\overline{q} = \overline{q}(\varepsilon, \eta, \hat{\eta}, \mathcal{X}, \hat{\mathcal{X}})$, $\overline{q} = (\overline{\nu}, 0)$, with $\vert \overline{q} \vert_{\mathbb{G}} \leq \varepsilon^{-1/4}$ such that for all  $\varepsilon \leq \varepsilon_1$ and all $(t, p)$:
    \begin{equation}\label{ineqLemmaRpart1}
        R^{\varepsilon}(t, p, \overline{q}, \eta, \mathcal{X}) \geq R^{*, \varepsilon}(t, p, \overline{q}, \hat{\eta}, \hat{\mathcal{X}}) - \varepsilon^2 h_K(\varepsilon^{1/4}),
    \end{equation}where $R^{*, \varepsilon}$ is defined as  $R^{\varepsilon}$ changing $\mathcal{F}$ by $\mathcal{F}^*$ and $ h_K(r):=\omega_{(1/2)K, R_0}(r),  \textnormal{ with } r \geq 0.$

    \item If $\Vert \hat{\eta} \Vert \leq K^{-1}$ ($K \in \mathbb{N}$), then there exists $\varepsilon_2 = \varepsilon_2(K, R_0, \lambda_0, \lambda_1)$ such that for all $(\eta, \mathcal{X}) \in \left( \mathbb{R}^{m_1} \backslash \lbrace 0 \rbrace \right) \times \mathcal{S}^{m_1}\left( \mathbb{R} \right)$, with $\Vert \eta \Vert \leq \varepsilon^{-1/4}, \Vert \mathcal{X} \Vert \leq \varepsilon^{-1/2}$, there exists $\overline{q} = \overline{q}(\varepsilon, \eta, \hat{\eta}, \mathcal{X}, \hat{\mathcal{X}})$, $\overline{q} = (\overline{\nu}, 0)$, with $\vert \overline{q} \vert_{\mathbb{G}} \leq \varepsilon^{-1/4}$ such that for  each $\varepsilon \leq \varepsilon_2$ and all $(t, p)$:
    \begin{equation}\label{ineqLemmaRpart2}
        R^{\varepsilon}(t, p, \overline{q}, \eta, \mathcal{X}) \geq R^{*, \varepsilon}(t, p, \overline{q}, 0, \hat{\mathcal{X}}).
    \end{equation}
    \end{enumerate}
\end{lemma}

\begin{proof}
  Assume that $\hat{\eta} \neq \eta$ and $\hat{\mathcal{X}} \neq \mathcal{X}$. Using orthonormal eigenvectores $\xi_0, \xi_1, \ldots, \xi_{m_1-1} \in \mathbb{R}^{m_1}$ of $\mathcal{X} - \hat{\mathcal{X}}$, we can represent $\nu$ with $\Vert \nu \Vert \leq \varepsilon^{-1/4}$ by:
    \begin{equation*}
            \nu= \sum_{i=0}^{m_1-1} s_i \xi_i,
    \end{equation*}
    where $s_i \in \mathbb{R}$ ($i=0, 1, \ldots, m_1-1$) with $s_0^2 + \ldots + s_{m_1-1}^2 \leq \varepsilon^{-1/4}$. In particular, let $\xi_0$ be the unit eigenvector which gives the maximum eigenvalue of $\hat{\mathcal{X}} - \mathcal{X}$. Thus $\varepsilon^{-2}[R^{\varepsilon}(t, p, \overline{q}, \eta, \mathcal{X}) - R^{*, \varepsilon}(t, p, \overline{q}, \hat{\eta}, \hat{\mathcal{X}})]$ is bounded from below by:
    \begin{equation}\label{equationvarepsilon2R}
        \begin{split}
        \varepsilon^{-1} s_0 \langle \hat{\eta} - \eta, \xi_0 \rangle &+ \varepsilon^{-1} \sum_{i=1}^{m_1-1} s_i \langle \hat{\eta} - \eta, \xi_i \rangle+ \dfrac{1}{2} s_0^2 \mathcal{E}\left( \hat{\mathcal{X}} - \mathcal{X} \right)  \\
        &+ \dfrac{1}{2} \sum_{i=1}^{m_1-1} s_i^2 \left\langle \left(\hat{\mathcal{X}} - \mathcal{X}\right) \xi_i, \xi_i \right\rangle + \left[ \mathcal{F} \left(t, p,  \hat{\eta}, \hat{\mathcal{X}} \right) - \mathcal{F} \left( t, p, \eta, \mathcal{X} \right) \right].
        \end{split}
    \end{equation}
    \begin{enumerate}
        \item Assume  $\Vert \hat{\eta} \Vert \geq K^{-1}$ for some $K \in \mathbb{N}$. 
        \begin{enumerate}
            \item If $\Vert \hat{\eta} - \eta \Vert \leq \varepsilon^{1/4}$, then $\Vert \eta \Vert \geq 1/2K$ for all sufficiently small $\varepsilon$. 
           In the case $\mathcal{E}\left(\hat{\mathcal{X}} - \mathcal{X} \right) > 0$, we take $\vert s_0 \vert = \lambda_1$, $s_i = 0$ for $i=1, \ldots, m_1-1$ in  (\ref{equationvarepsilon2R}).  Then (\ref{equationvarepsilon2R}) is rewritten by:
            \begin{equation}\label{equationlemmaR}
            \begin{split}
            \varepsilon^{-1} \lambda_1 \vert \langle \hat{\eta} - \eta, &\xi_0 \rangle \vert + \dfrac{\lambda_1^2}{2} \mathcal{E}^+ \left( \hat{\mathcal{X}}  -\mathcal{X}  \right) + \left[ \mathcal{F} \left( t, p, \hat{\eta}, \hat{\mathcal{X}} \right) - \mathcal{F} \left( t, p, \hat{\eta}, \mathcal{X} \right) \right] + \left[ \mathcal{F} \left( t, p, \hat{\eta}, \mathcal{X} \right) - \mathcal{F} \left(t, p, \eta, \mathcal{X} \right) \right],
            \end{split}
            \end{equation}
            
            where we choose an appropriate sign of $s_0$ so that $s_0 \langle \hat{\eta} - \eta, \xi_0 \rangle$ is non-negative. From \textbf{($\mathcal{F}3$)},  for any $\eta \in \mathbb{R}^{m_1} \backslash \lbrace 0 \rbrace$:
            \begin{equation}\label{ineq1lemmaR}
                \dfrac{\lambda_1^2}{2}\mathcal{E}^+ \left( \mathcal{X} - \hat{\mathcal{X}} \right) + \left[ \mathcal{F} \left( t, p, \hat{\eta}, \hat{\mathcal{X}} \right) - \mathcal{F} \left(t, p,  \hat{\eta}, \mathcal{X} \right) \right] \geq 0.
            \end{equation}
            From \textbf{($\mathcal{F}4$)} and  $\Vert \eta \Vert \geq 1/2K$ it follows: \begin{equation}\label{ineq2lemmaR}
                \begin{split}
                    &\mathcal{F}(t, p, \hat{\eta}, \mathcal{X}) - \mathcal{F}(t, p, \eta, \mathcal{X}) \geq - \omega_{1/2K, R_0}(\varepsilon^{1/4})
                 \end{split}
                \end{equation}
            where $\omega_{1/2K, R_0}$ is a modulus of continuity depending only on $K$ and $R_0$. Thus, from \eqref{ineq1lemmaR}, \eqref{ineq2lemmaR} and \eqref{equationlemmaR} we obtain:
            \begin{equation}\label{ineqlemmaR}
            \varepsilon^{-2}[R^{\varepsilon}(t, p, \overline{q}, \eta, \mathcal{X}) - R^{*, \varepsilon}(t, p, \overline{q}, \hat{\eta}, \hat{\mathcal{X}})] \geq - h_K(\varepsilon^{1/4}),
            \end{equation}
            where $h_K(s) = \omega_{1/2K, R_0}(s)$.
           
            If now  $\mathcal{E}\left(\hat{\mathcal{X}} - \mathcal{X} \right) \leq 0$, we take $s_i=0$ for $i=0, 1, \ldots, m_1 - 1$ in the formula \eqref{equationvarepsilon2R}. Then $\mathcal{F}(t, p, \hat{\eta}, \hat{\mathcal{X}}) \geq \mathcal{F}(t, p, \hat{\eta}, \mathcal{X})$ for any $\eta \in \mathbb{R}^{m_1} \backslash \lbrace 0 \rbrace$ holds, since $\mathcal{F}$ is degenerate elliptic (see Remark \ref{remarkhipotesysofFandH}). From \eqref{ineq2lemmaR}, the inequality \eqref{ineqlemmaR} is derived (with the same modulus $h_K$).

            \item If $\Vert \hat{\eta} - \eta \Vert \geq \varepsilon^{1/4}$, then:
            \begin{equation}\label{etarepresentation}
                \dfrac{\hat{\eta} - \eta}{\Vert \hat{\eta} - \eta \Vert} = \sum\limits_{i=0}^{m_1 - 1} r_i \xi_i,
            \end{equation}
            where $r_i \in \mathbb{R}$ with $r_0^2 + r_1^2 + \ldots + r_{m_1 - 1}^2 = 1$. Let us divide this case into two parts. 
           
            Suppose first that  $\vert \langle \hat{\eta} - \eta, \xi_0 \rangle \vert \geq  \dfrac{\varepsilon^{1/2}}{\lambda_1}$. Then if  $\mathcal{E}(\hat{\mathcal{X}} - \mathcal{X}) > 0$,  we choose $s_i$ so that $\vert s_0 \vert = \lambda_1$, $s_i = 0$ ($i= 1, \ldots, m_1 - 1$) and obtain a formula similar to \eqref{equationlemmaR}:
                \begin{equation*}
                \begin{split}
                    \varepsilon^{-1} \lambda_1 &\vert \langle \hat{\eta} - \eta, \xi_0 \rangle \vert + \dfrac{\lambda_1^2}{2} \mathcal{E}^+ \left(  \hat{\mathcal{X}} -  \mathcal{X}\right) \\
                    &+ \left[ \mathcal{F} \left(t, p,  \eta, \hat{\mathcal{X}} \right) - \mathcal{F} \left(t, p,  \eta, \mathcal{X} \right) \right] + \left[ \mathcal{F} \left(t, p,  \hat{\eta}, \hat{\mathcal{X}} \right) - \mathcal{F} \left(t, p, \eta, \hat{\mathcal{X}} \right)\right].
                \end{split}
                \end{equation*}
               Similarly, by  \textbf{($\mathcal{F}3$)}
             and Remark \ref{remarkhipotesysofFandH}, we have for all $\varepsilon \leq \varepsilon(R_0, \lambda_0, \lambda_1)$:

            \begin{equation*}
            \begin{split}
            \varepsilon^{-2}[R^{\varepsilon}(t, p, \overline{q}, \eta, \mathcal{X}) &- R^{*, \varepsilon}(t, p, \overline{q}, \hat{\eta}, \hat{\mathcal{X}})] \\
            &\geq \varepsilon^{-1/2} - 2C(1 + \Vert \hat{\mathcal{X}} \Vert) \\
            &\geq \varepsilon^{-1/4} - 2C(1 + R_0) \\         
            &\geq 0.
            \end{split}
            \end{equation*}

           If $\mathcal{E}( \hat{\mathcal{X}}- \mathcal{X}) \leq 0$, we choose $s_i$ so that $\vert s_0 \vert = \varepsilon^{1/4} \lambda_1$, $s_i=0$ ($i=1, \ldots, m_1 - 1$), and obtain a similar formula  to \eqref{equationlemmaR}. Hence, there is $\varepsilon_1(R_0, K, \lambda_0, \lambda_1)$ so that:

            \begin{equation*}
            \begin{split}
                \varepsilon^{-2}[R^{\varepsilon}(t, p, \overline{q}, &\eta, \mathcal{X}) - R^{*, \varepsilon}(t, p, \overline{q}, \hat{\eta}, \hat{\mathcal{X}})] \\
                &\geq \varepsilon^{-1/2} + \dfrac{\lambda_1^2}{2}\varepsilon^{1/2} \mathcal{E}(\hat{\mathcal{X}} - \mathcal{X})  - 2C(1 + \Vert \hat{\mathcal{X}} \Vert) \\
                &\geq \varepsilon^{-1/4} - \dfrac{\lambda_1^2}{2} \left( \varepsilon^{1/2}R_0 + 1 \right) - 2C(1 + R_0) \\
                &\geq \varepsilon^{-1/4} - \lambda_1^2 - 2C(1 + R_0)\\
                &\geq 0,
            \end{split}
            \end{equation*}

             We consider now the case $\vert \langle \hat{\eta} - \eta, \xi_0 \rangle \vert \leq  \dfrac{\varepsilon^{1/2}}{\lambda_1}$. Then from \eqref{etarepresentation}:
            \begin{equation}\label{r cero}
                \vert r_0 \vert = \Big\vert \dfrac{\langle \hat{\eta} - \eta, \xi_0 \rangle}{\Vert \hat{\eta} - \eta \Vert} \Big\vert \leq  \dfrac{\varepsilon^{1/4}}{\lambda_1} =: c_0 \varepsilon^{1/4}.
            \end{equation}
            Since $r_0^2 + r_1^2 + \ldots + r_{m_1 - 1}^2 = 1$, we have the inequality:
            \begin{equation*}
                1 - c_0^2 \varepsilon^{1/2} \leq  1 - r_0^2 = r_1^2 + r_2^2 + \ldots + r_{m_1 -1}^2 \leq \vert r_1 \vert + \vert r_2 \vert + \ldots + \vert r_{m_1 -1} \vert,
            \end{equation*}
            where we take $\varepsilon$ so that $c_0^2 \varepsilon^{1/2} < 1/2$. This inequality implies that there exists at least one number $j_0$ such that:
            \begin{equation*}
                \vert r_{j_0} \vert \geq \dfrac{1 - c_0^2 \varepsilon^{1/2}}{m_1 - 1} > \dfrac{1}{2(m_1 - 1)}.
            \end{equation*}
            Now we take $s_i$ so that $s_i = 0$ ($i \neq 0, j_0$) in  \eqref{equationvarepsilon2R} to get:
            \begin{equation}\label{equationvarepsilon2Rj0}
            \begin{split}
                \varepsilon^{-1} s_0 \langle \hat{\eta} - \eta, \xi_0 \rangle &+ \varepsilon^{-1} s_{j_0} \langle \hat{\eta} - \eta, \xi_{j_0} \rangle + \dfrac{s_0^2}{2}  \mathcal{E}\left( \hat{\mathcal{X}} - \mathcal{X} \right) + \dfrac{s_{j_0}^2}{2} \left\langle \left( \hat{\mathcal{X}} - \mathcal{X} \right) \xi_{j_0}, \xi_{j_0} \right\rangle \\
                &+ \left[ \mathcal{F} \left(t, p, \eta, \hat{\mathcal{X}} \right) - \mathcal{F} \left( t, p, \eta, \mathcal{X} \right) \right] + \left[ \mathcal{F} \left( t, p, \hat{\eta}, \hat{\mathcal{X}} \right) - \mathcal{F} \left(t, p,  \eta, \hat{\mathcal{X}} \right) \right].
            \end{split}
            \end{equation}
            If $\mathcal{E}(\hat{\mathcal{X}}-\mathcal{X}) > 0$, we chose $\vert s_0 \vert = \lambda_1$ with $s_0 \langle \hat{\eta} - \eta, \xi_0 \rangle \geq 0$. In addition, take $\vert s_{j_0} \vert = \lambda_1 \varepsilon^{1/4}$ so that  $s_{j_0} \langle \hat{\eta} - \eta, \xi_{j_0} \rangle \geq 0$. Then, using also \eqref{r cero},  \eqref{equationvarepsilon2Rj0} is rewritten as:
                \begin{equation*}
                \begin{split}
                    \varepsilon^{-1} \lambda_1 \vert r_0 \vert \Vert \hat{\eta} - \eta \Vert &+ \varepsilon^{-3/4} \lambda_1 \vert r_{j_0} \vert \Vert \hat{\eta} - \eta \Vert+ \dfrac{\lambda_1^2}{2}  \mathcal{E}\left( \hat{\mathcal{X}} - \mathcal{X} \right) + \dfrac{\lambda_1^2}{2} \varepsilon^{1/2} \left\langle \left( \hat{\mathcal{X}} - \mathcal{X} \right) \xi_{j_0}, \xi_{j_0} \right\rangle \\
                    &+ \left[ \mathcal{F} \left(t, p,  \eta, \hat{\mathcal{X}} \right) - \mathcal{F} \left(t, p,  \eta, \mathcal{X} \right) \right] + \left[ \mathcal{F} \left(t, p,  \hat{\eta}, \hat{\mathcal{X}} \right) - \mathcal{F} \left(t, p,  \eta, \hat{\mathcal{X}} \right) \right].
                \end{split}
                \end{equation*}  
               Appealing to \textbf{($\mathcal{F}3$)} and  Remark \ref{remarkhipotesysofFandH},  there exists $\epsilon_1(R_0, K, \lambda_0, \lambda_1)$ so that:
                \begin{equation*}
                \begin{split}
                    \varepsilon^{-2}[R^{\varepsilon}(\overline{q}, &\eta, \mathcal{X}) - R^{*, \varepsilon}(\overline{q}, \hat{\eta}, \hat{\mathcal{X}})] \\
                    &\geq \varepsilon^{-3/4} \lambda_1 \vert r_{j_0} \vert \Vert \hat{\eta} - \eta \Vert + \dfrac{\lambda_1^2}{2} \varepsilon^{1/2} \left\langle \left( \hat{\mathcal{X}} - \mathcal{X} \right) \xi_{j_0}, \xi_{j_0} \right\rangle + \left[ \mathcal{F} \left(p, t,  \hat{\eta}, \hat{\mathcal{X}} \right) - \mathcal{F} \left(t, p,  \eta, \hat{\mathcal{X}} \right) \right]\\
                    &\geq \dfrac{\lambda_1 \varepsilon^{-1/2}}{2(m_1 - 1)} - \dfrac{\lambda_1^2}{2}\varepsilon^{1/2} \Vert  \hat{\mathcal{X}} - \mathcal{X} \Vert - 2C(1 + \Vert \hat{\mathcal{X}} \Vert)\geq 0.
                \end{split}
                \end{equation*}

               In the case $\mathcal{E}(\hat{\mathcal{X}}- \mathcal{X} ) \leq 0$, take $s_0= 0$ and $\vert s_{j_0} \vert = \lambda_1 \varepsilon^{1/4}$ so that $s_{j_0} \langle \hat{\eta} - \eta, \xi_{j_0} \rangle \geq 0$. Then, as in the previous case we have: 
                \begin{equation*}
                \begin{split}
                    \varepsilon^{-3/4} \lambda_1 &\vert r_{j_0} \vert \Vert \hat{\eta} - \eta \Vert + \dfrac{\lambda_1^2}{2} \varepsilon^{1/2} \left\langle \left( \hat{\mathcal{X}} - \mathcal{X} \right) \xi_{j_0}, \xi_{j_0} \right\rangle \\
                    &+ \left[ \mathcal{F} \left(t, p,  \eta, \hat{\mathcal{X}} \right) - \mathcal{F} \left(t, p,  \eta, \mathcal{X} \right) \right] + \left[ \mathcal{F} \left( t, p, \hat{\eta}, \hat{\mathcal{X}} \right) - \mathcal{F} \left( t, p, \eta, \hat{\mathcal{X}} \right) \right]\geq 0.
                \end{split}
                \end{equation*}
                
            \end{enumerate}

        In particular, since $\Vert \hat{\eta} \Vert \geq K^{-1}$, we see $\mathcal{F}(t, p, \hat{\eta}, \mathcal{X}) = \mathcal{F}^*(t, p, \hat{\eta}, \mathcal{X})$. Consequently if we set $\varepsilon_1 = \varepsilon_1(K, R_0, \lambda_0, \lambda_1)$, then the formula \eqref{ineqLemmaRpart1} holds with $h_K(s)= \omega_{R_0, 1/K}(s)$.
        
        \item For the case $\Vert \hat{\eta} \Vert \leq K^{-1}$ for $K \in \mathbb{N}$,  we argue as in \textbf{Case 1} to derive the  estimate \eqref{ineqLemmaRpart2}.
        
    \end{enumerate}

Finally, we consider the case of $\eta = \hat{\eta}$ or $\mathcal{X} = \hat{\mathcal{X}}$ for $\hat{\eta} \in \mathbb{R}^{m_1} \backslash \lbrace 0 \rbrace$. We can choose sequences $\lbrace \eta_k \rbrace \subset \mathbb{R}^{m_1} \backslash \lbrace 0 \rbrace$ and $\lbrace \mathcal{X}_n \rbrace \subset \mathcal{S}^{m_1}(\mathbb{R})$ such that $\eta_k \neq \eta$, $\eta_k \to \hat{\eta}$, $\mathcal{X}_k \neq \mathcal{X}$, $\mathcal{X}_n \to \hat{\mathcal{X}}$ as $k, n \to \infty$, respectively. Now let us set $\overline{q}_{k, n} := \overline{q}(\varepsilon, \eta_k, \hat{\eta}, \mathcal{X}_n, \hat{\mathcal{X}})$ where $\overline{q}(\varepsilon, \eta_k, \hat{\eta}, \mathcal{X}_n, \hat{\mathcal{X}})$ satisfies the inequality \eqref{ineqLemmaRpart1} or \eqref{ineqLemmaRpart2}. As $\vert \overline{q}_{k, n} \vert_\mathbb{G} \leq \varepsilon^{-1/4}$, by compactness, the conclusion follows by taking   as $k \to \infty$ and then $n \to \infty$ in  \eqref{ineqLemmaRpart1} and  \eqref{ineqLemmaRpart2}).
\end{proof}

\end{document}